\newcommand{\Aut}{\mathop{\mathrm{Aut}}\nolimits}
\newcommand{\End}{\mathop{\mathrm{End}}\nolimits}
\newcommand{\Hom}{\mathop{\mathrm{Hom}}\nolimits}

\newcommand{\Sy}{\mathop{\mathfrak{S}}\nolimits}
\documentclass[11pt,twoside]{amsart} 
\usepackage{amssymb}
\advance \topmargin by -\headheight
\advance \topmargin by -\headsep
\evensidemargin \oddsidemargin
\newtheorem{theorem}{Theorem}[section] 
\newtheorem{conjecture}[theorem]{Conjecture}
\newtheorem{corollary}[theorem]{Corollary}

\newtheorem{lemma}[theorem]{Lemma}
\newtheorem{proposition}[theorem]{Proposition}
\theoremstyle{remark} 
\newtheorem{example}[theorem]{Example}
\theoremstyle{remark} 
\newtheorem{remark}[theorem]{Remark}
\begin{document} 
\title{An analogue of Hilbert's Theorem 90 for infinite symmetric groups} 
\author{M.Rovinsky} 
\address{National Research University Higher School of Economics, 
AG Laboratory, HSE, 6 Usacheva str., Moscow, Russia, 119048 
\& Institute for Information Transmission Problems of Russian 
Academy of Sciences}
\email{marat@mccme.ru}
\date{}
\begin{abstract} Let $K$ be a field and $G$ be a group of its automorphisms. 

If $G$ is precompact then $K$ is a generator of the category of {\sl smooth} (i.e. with 
open stabilizers) $K$-{\sl semilinear} representations of $G$, cf. Proposition \ref{Satz1}. 

There are non-semisimple smooth semilinear representations of $G$ over $K$ if $G$ is not precompact. 

In this note the smooth semilinear representations of the group $\Sy_{\Psi}$ of all 
permutations of an infinite set $\Psi$ are studied. Let $k$ be a field and $k(\Psi)$ 
be the field freely generated over $k$ by the set $\Psi$ (endowed with the natural 
$\Sy_{\Psi}$-action). One of principal results describes the Gabriel spectrum of 
the category of smooth $k(\Psi)$-semilinear representations of $\Sy_{\Psi}$. 

It is also shown, in particular, that (i) for any smooth $\Sy_{\Psi}$-field $K$ any 
smooth finitely generated $K$-semilinear representation of $\Sy_{\Psi}$ is noetherian, 
(ii) for any $\Sy_{\Psi}$-invariant subfield $K$ in the field $k(\Psi)$, the object $k(\Psi)$ 
is an injective cogenerator of the category of smooth $K$-semilinear representations of $\Sy_{\Psi}$, 
(iii) if $K\subset k(\Psi)$ is the subfield of rational homogeneous functions of degree 
0 then there is a one-dimensional $K$-semilinear representation of $\Sy_{\Psi}$, whose 
integral tensor powers form a system of injective cogenerators of the category of smooth 
$K$-semilinear representations of $\Sy_{\Psi}$, 
(iv) if $K\subset k(\Psi)$ is the subfield generated over $k$ by $x-y$ for all 
$x,y\in\Psi$ then there is a unique isomorphism class of indecomposable smooth 
$K$-semilinear representations of $\Sy_{\Psi}$ of each given finite length. 
\end{abstract} 

\maketitle 
\section{Introduction} 
\subsection{Goals} Let $G$ be a group of permutations of a set $C$. Then the group $G$ is endowed 
with the topology, whose base is given by the translates of the pointwise stabilizers 
of the finite subsets in $C$. From now on, $C$ is a field and $G$ consists of field 
automorphisms. We are interested in continuous $G$-actions on discrete sets (i.e., with 
open stabilizers), called {\sl smooth} in what follows. These $G$-sets will be vector spaces 
over $G$-invariant subfields $K\subseteq C$, while the $G$-actions will be {\sl semilinear}. 

\subsection{Motivation} \label{motiv} The problem of describing certain irreducible smooth 
semilinear representations of $G$ in $C$-vector spaces arises in some algebro-geometric 
problems, where $C$ is an algebraically closed extension of infinite transcendence degree 
of an algebraically closed field $k$ of characteristic 0 and $G$ is the group of all 
automorphisms of the field $C$ leaving $k$ fixed. This is briefly explained in \S\ref{prehist}. 

Fix a transcendence base $\Psi$ of $C$ over $k$ and denote by $k(\Psi)$ the subfield of 
$C$ generated over $k$ by the set $\Psi$. Then taking invariants of the Galois group of the 
extension $C|k(\Psi)$ is a faithful and exact functor from the category of smooth semilinear 
representations of $G$ over $C$ to the category of smooth semilinear representations 
over $k(\Psi)$ of the group $\Sy_{\Psi}$ of all permutations of the set $\Psi$. 

Then the problem splits into two parts: (i) to describe the smooth $k(\Psi)$-semilinear 
representations of $\Sy_{\Psi}$ and (ii) to relate smooth $k(\Psi)$-semilinear representations 
of $\Sy_{\Psi}$ and `interesting' smooth $k$-linear and $C$-semilinear representations of $G$. 
We study (i) in detail and give some remarks on (ii). 

\subsection{Basic notation} 
For an abelian group $A$ and a set $S$ we denote by $A\langle S\rangle$ the abelian group, 
which is the direct sum of copies of $A$ indexed by $S$, i.e., the elements of 
$A\langle S\rangle$ are the finite formal sums $\sum_{i=1}^Na_i[s_i]$ for all integer $N\ge 0$, 
$a_i\in A$, $s_i\in S$, with addition defined obviously. In some cases, $A\langle S\rangle$ 
will be endowed with an additional structure, e.g., of a module, a ring, etc. 

If $A$ is an associative ring endowed with an action of a group $G$ respecting both 
operations in $A$, we consider $A\langle G\rangle$ as a unital associative ring with 
the unique multiplication such that $(a[g])(b[h])=ab^g[gh]$, where we write $a^h$ for 
the result of applying of $h\in G$ to $a\in A$. 

The left $A\langle G\rangle$-modules are 
also called $A$-{\sl semilinear representations} of $G$ if $A$ is a field. 

\subsection{Hilbert's theorem 90} 
Let now $K$ be a field and $G$ be a group of its automorphisms. Then Speiser's 
generalization of Hilbert's theorem 90, cf. \cite[Satz 1]{Speiser}, or 
\cite[Prop. 3, p.159]{CL}, can be interpreted and slightly generalized further as follows. 
\begin{proposition} \label{Satz1} The category of smooth $K$-semilinear 
representations of $G$ admits a simple generator if and only if $G$ is 
precompact, i.e., any open subgroup of $G$ is of finite index. \end{proposition} 
\begin{proof} If the category of smooth $K$-semilinear representations 
of $G$ admits a simple generator then any object is a direct sum of copies 
of $K$, in particular, semisimple. This implies that $G$ is precompact, since otherwise 
there is an open subgroup $U\subset G$ of infinite index, while the representation 
$K\langle G/U\rangle$ of $G$ has no non-zero vectors fixed by $G$, unlike its simple quotient $K$. 
(For a $G$-set $S$ we consider $K\langle S\rangle$ as a $K$-vector space with the diagonal $G$-action.) 

If $G$ is finite then \cite[Satz 1]{Speiser}, appropriately reformulated, implies that 
any $K$-semilinear representation of $G$ is a sum of copies of $K$. Namely, with 
$k:=K^G$, the field extension $K|k$ is finite, so the natural $G$-action on $K$ gives 
rise to a $k$-algebra homomorphism $K\langle G\rangle\to\mathrm{End}_k(K)$, which is (a) 
surjective by Jacobson's density theorem and (b) injective by independence of characters. 
Then any $K\langle G\rangle$-module is isomorphic to a direct sum of copies of $K$. 

For arbitrary precompact $G$, a smooth $K$-semilinear representation $V$ of $G$ 
and $v\in V$ the intersection $H$ of all conjugates of the stabilizer of $v$ in 
$G$ is of finite index. Thus, $v$ is contained in the $K^H$-semilinear 
representation $V^H$ of the group $G/H$. As $G/H$ is finite, 
$V^H=(V^H)^{G/H}\otimes_{(K^H)^{G/H}}K^H=V^G\otimes_{K^G}K^H$, i.e., $v$ is contained 
in a subrepresentation isomorphic to a direct sum of copies of $K$. \end{proof} 

\subsection{Results} \label{NotaF} Let $k$ be a field and $\{A_i\}_{i\in\Psi}$ be a collection 
of unital associative $k$-algebras, indexed by a set $\Psi$. Denote by $\bigotimes_{k,~i\in\Psi}A_i$ 
the coproduct of $k$-algebras, i.e., $\bigotimes_{k,~i\in\Psi}A_i:=
\mathop{\underrightarrow{\lim}}\limits_{I\subset\Psi}\bigotimes_{k,~i\in I}A_i$ is inductive 
limit of the system of the tensor products $\bigotimes_{k,~i\in I}A_i$ over $k$ for all finite 
subsets $I\subset\Psi$, consisting of finite linear combinations of tensor products 
of elements in $A_i$, almost all equal to 1. 

Let $F$ be a field and $k$ be a subfield algebraically closed in $F$. Denote by 
$F_{\Psi}=F_{k,\Psi}$ the field of fractions of the $k$-algebra $\bigotimes_{k,~i\in\Psi}F$. 
The group $\Sy_{\Psi}$ of all permutations of the set $\Psi$ acts on $\bigotimes_{k,~i\in\Psi}F$ 
by permuting the tensor factors, and thus, it acts on the field $F_{\Psi}$. 

For instance, if $F=k(x)$ is the field of rational functions in one variable then 
$F_{\Psi}=k(\Psi)$ is the field of rational functions over $k$ in the variables enumerated 
by the set $\Psi$, while the group $\Sy_{\Psi}$ acts on $k(\Psi)$ by permuting the variables. 

\begin{theorem} \label{Gabriel-spectrum} Let $\Psi$ be a set, $F$ be a field 
and $k$ be a subfield algebraically closed in $F$. 

Assume that transcendence degree of the field extension $F|k$ is at most continuum. 

Let $K\subseteq F_{\Psi}$ be an $\Sy_{\Psi}$-invariant subfield. Then the object $F_{\Psi}$ is 
an injective cogenerator of the category of smooth $K\langle\Sy_{\Psi}\rangle$-modules. 

In particular, {\rm (i)} any smooth $K\langle\Sy_{\Psi}\rangle$-module 
can be embedded into a direct product of copies of $F_{\Psi}$; {\rm (ii)} any smooth 
$F_{\Psi}\langle\Sy_{\Psi}\rangle$-module 
of finite length is isomorphic to a direct sum of copies of $F_{\Psi}$. 

Gabriel spectrum of the category of smooth $F_{\Psi}\langle\Sy_{\Psi}\rangle$-modules, 
i.e., the set of isomorphism classes of indecomposable injectives, 
consists of $F_{\Psi}\langle\binom{\Psi}{s}\rangle$ for all integer $s\ge 0$, 
where $\binom{\Psi}{s}$ denotes the set of all subsets of $\Psi$ of cardinality $s$. The closure of 
$F_{\Psi}\langle\binom{\Psi}{s}\rangle$ is the set 
$\{F_{\Psi},F_{\Psi}\langle\Psi\rangle,\dots,F_{\Psi}\langle\binom{\Psi}{s}\rangle\}$. \end{theorem} 
Theorem~\ref{Gabriel-spectrum}, can be considered as an example of a field $K$ 
and a non-precompact group $G$ of its automorphisms such that the smooth {\sl irreducible} 
$K$-semilinear representations of $G$ admit an explicit description. In Theorems 
\ref{smooth-simple-degree-0} and \ref{unip} two more examples are presented with the same 
group $G=\Sy_{\Psi}$ showing that description depends crucially on the field $K$. 

For each $d\in\mathbb Z$, let $V_d\subseteq k(\Psi)$ be the subset of homogeneous 
rational functions of degree $d$, so $V_0$ is an $\Sy_{\Psi}$-invariant subfield and 
$V_d\subseteq k(\Psi)$ is an $\Sy_{\Psi}$-invariant one-dimensional $V_0$-vector subspace. 
\begin{theorem} \label{smooth-simple-degree-0} The objects $V_d$ for 
$d\in\mathbb Z$ form a system of injective cogenerators of the category of smooth 
$V_0$-semilinear representations of $\Sy_{\Psi}$, i.e., any smooth $V_0$-semilinear 
representation $V$ of $\Sy_{\Psi}$ can be embedded into a direct product of cartesian 
powers of $V_d$. In particular, any smooth $V_0$-semilinear representation of $\Sy_{\Psi}$ 
of finite length is isomorphic to $\bigoplus_{d\in\mathbb Z}V_d^{m(d)}$ for a unique, 
if the set $\Psi$ is infinite, function $m:\mathbb Z\to\mathbb Z_{\ge 0}$ with finite 
support. \end{theorem} 

\begin{remark} \label{irred-contained-irred-semilin} Let $K$ be a field and 
$G$ be a group of automorphisms of $K$. Let $k\subseteq K^G$ be a subfield. 
Then any smooth irreducible representation $W$ of $G$ over $k$ can 
be embedded into a smooth irreducible $K$-semilinear representation of $G$. 
Indeed, $W$ can be embedded into any irreducible quotient of the $K$-semilinear 
representation $W\otimes_kK$. 
\end{remark} 

\begin{corollary} In the above notation, any 
smooth irreducible representation of $\Sy_{\Psi}$ over a field $k$ can be embedded 
into the $V_0$-semilinear representation $V_d\subset k(\Psi)$ for some integer $d$. 
\end{corollary} 
This follows from Remark \ref{irred-contained-irred-semilin} and 
Theorem \ref{smooth-simple-degree-0}. \qed 

\begin{theorem} \label{unip} Suppose that the set $\Psi$ is infinite. Let 
$K\subset k(\Psi)$ be the subfield generated over $k$ by the rational functions 
$x-y$ for all $x,y\in\Psi$, so the group $\Sy_{\Psi}$ acts naturally on the 
fields $k(\Psi)$ and $K$. Fix some $x\in\Psi$. Then {\rm (i)} the injective envelope of $K$ 
in the category of smooth $K$-semilinear representations of $\Sy_{\Psi}$ is isomorphic to 
$K[x]$; {\rm (ii)} object $K[x]$ of is an indecomposable cogenerator 
of the category of smooth $K$-semilinear representations of $\Sy_{\Psi}$; {\rm (iii)} for 
any integer $N\ge 1$ there exists a unique isomorphism class of smooth $K$-semilinear 
indecomposable representations of $\Sy_{\Psi}$ of length $N$. \end{theorem} 

Finally, Theorem~\ref{locally-noetherian} asserts that, for any left noetherian associative 
ring $A$ endowed with a smooth $\Sy_{\Psi}$-action, the category of smooth left 
$A\langle\Sy_{\Psi}\rangle$-modules is locally noetherian. 

\section{Open subgroups and permutation modules} 
For any set $\Psi$ and a subset $T\subseteq\Psi$, we denote by 
$\Sy_{\Psi|T}$ the pointwise stabilizer of $T$ in the group $\Sy_{\Psi}$. 
Let $\Sy_{\Psi,T}:=\Sy_{\Psi\smallsetminus T}\times\Sy_T$ denote the 
group of all permutations of $\Psi$ preserving $T$ (in other words, 
the setwise stabilizer of $T$ in the group $\Sy_{\Psi}$, or 
equivalently, the normalizer of $\Sy_{\Psi|T}$ in $\Sy_{\Psi}$). 

\begin{lemma} \label{strong-generation} For any pair of finite subsets 
$T_1,T_2\subset\Psi$ the subgroups $\Sy_{\Psi|T_1}$ and $\Sy_{\Psi|T_2}$ 
generate the subgroup $\Sy_{\Psi|T_1\cap T_2}$. \end{lemma}
\begin{proof} Let us show first that 
$\Sy_{\Psi|T_1}\Sy_{\Psi|T_2}=\{g\in\Sy_{\Psi|T_1\cap T_2}~|~
g(T_2)\cap T_1=T_1\cap T_2\}=:\Xi$. The inclusion $\subseteq$ is trivial. 
On the other hand, \[\Xi/\Sy_{\Psi|T_2}=\{\text{embeddings 
$T_2\smallsetminus(T_1\cap T_2)\hookrightarrow\Psi\smallsetminus T_1$}\},\] 
while the latter is an $\Sy_{\Psi|T_1}$-orbit. \end{proof} 

\begin{lemma} \label{open-subgrps-descr} For any open subgroup $U$ of 
$\Sy_{\Psi}$ there exists a unique subset $T\subset\Psi$ such that 
$\Sy_{\Psi|T}\subseteq U$ and the following equivalent conditions 
hold: {\rm (a)} $T$ is minimal; {\rm (b)} $\Sy_{\Psi|T}$ is normal 
in $U$; {\rm (c)} $\Sy_{\Psi|T}$ is of finite index in $U$. 
In particular, {\rm (i)} such $T$ is finite, {\rm (ii)} the 
open subgroups of $\Sy_{\Psi}$ correspond bijectively to 
the pairs $(T,H)$ consisting of a finite subset $T\subset\Psi$ 
and a subgroup $H\subseteq\mathrm{Aut}(T)$ under $(T,H)\mapsto
\{g\in\Sy_{\Psi,T}~|~\text{{\rm restriction of $g$ to $T$ 
belongs to $H$}}\}$. \end{lemma} 
\begin{proof} Any open subgroup $U$ in $\Sy_{\Psi}$ contains 
the subgroup $\Sy_{\Psi|T}$ for a finite subset $T\subset\Psi$. 
Assume that $T$ is chosen to be minimal. If $\sigma\in U$ then 
$U\supseteq\sigma\Sy_{\Psi|T}\sigma^{-1}=\Sy_{\Psi|\sigma(T)}$, and 
therefore, (i) $\sigma(T)$ is also minimal, (ii) $U$ contains the 
subgroup generated by $\Sy_{\Psi|\sigma(T)}$ and $\Sy_{\Psi|T}$. 
By Lemma \ref{strong-generation}, the subgroup generated by 
$\Sy_{\Psi|\sigma(T)}$ and $\Sy_{\Psi|T}$ is $\Sy_{\Psi|T\cap\sigma(T)}$, 
and thus, $U$ contains the subgroup $\Sy_{\Psi|T\cap\sigma(T)}$. 
The minimality of $T$ means that $T=\sigma(T)$, i.e., 
$U\subseteq\Sy_{\Psi,T}$. If $T'\subset\Psi$ is another minimal 
subset such that $\Sy_{\Psi|T'}\subseteq U$ then, by Lemma 
\ref{strong-generation}, $\Sy_{\Psi|T\cap T'}\subseteq U$, so 
$T=T'$, which proves (b) and (the uniqueness in the case) (a). 
It follows from (b) that $\Sy_{\Psi|T}\subseteq U\subseteq\Sy_{\Psi,T}$, 
so $\Sy_{\Psi|T}$ is of finite index in $U$. As the subgroups 
$\Sy_{\Psi|T}$ and $\Sy_{\Psi|T'}$ are not commensurable for 
$T'\neq T$, we get the uniqueness in the case (c). \end{proof} 

\begin{lemma} \label{no-simple-submod} Let $G$ be a group acting on a field 
$K$ and $K'\subseteq K$ be a $G$-invariant subfield such that any simple 
$K'\langle G\rangle$-submodule in $K$ is isomorphic to $K'$. Let $U\subset G$ 
be a subgroup such that an element $g\in G$ acts identically on $K^U$ if and 
only if $g\in U$. Then there are no irreducible $K'$-semilinear subrepresentations 
in $K\langle G/U\rangle$, unless $U$ is of finite index in $G$. If $G$ acts faithfully on $K$ 
and $U$ is of finite index in $G$ then $K\langle G/U\rangle$ is trivial. 

If $G=\Sy_{\Psi}$ and $U\subset\Sy_{\Psi}$ is a proper open subgroup then 
{\rm (i)} index of $U$ in $\Sy_{\Psi}$ is infinite; {\rm (ii)} there 
are no elements in $\Sy_{\Psi}\smallsetminus U$ acting identically on 
$K^U$ if the $\Sy_{\Psi}$-action on $K$ is non-trivial. \end{lemma} 

{\sc Example and notation.} Let $G$ be a group acting on a field $K=K'$; 
$U\subset G$ be a maximal proper subgroup. Assume that $K^U\neq K^G=:k$. 
Then we are under assumptions of Lemma \ref{no-simple-submod}. 

The representation $K\langle G/U\rangle$ is highly reducible: any finite-dimensional 
$K^G$-vector subspace $\Xi$ in $K^U$, determines a surjective morphism 
$K\langle G/U\rangle\to\Hom_k(\Xi,K)$, $[g]\mapsto[Q\mapsto Q^g]$, which is surjective, 
since $K^U=\Hom_{K\langle\Sy_{\Psi}\rangle}(K\langle G/U\rangle,K)$ under $Q:[g]\mapsto gQ$. 

More particularly, let $G=\Sy_{\Psi}$. Let $U\subset\Sy_{\Psi}$ be a maximal proper 
subgroup, i.e., $U=\Sy_{\Psi,I}$ for a finite subset $I\subset\Psi$ (so $\Sy_{\Psi}/U$ can 
be identified with the set $\binom{\Psi}{\#I}$). Suppose that $K^{\Sy_{\Psi,I}}\neq k$. 
Then we are under assumptions of Lemma \ref{no-simple-submod}, so there are no 
irreducible $K$-semilinear subrepresentations in $K\langle\binom{\Psi}{\#I}\rangle$. 

\begin{proof} The elements $[g]\in G/U$ can be considered as certain pairwise 
distinct one-dimensional characters $\chi_{[g]}:(K^U)^{\times}\to K^{\times}$. 
By Artin's independence of characters theorem, the characters 
$\chi_{[g]}$ are linearly independent in the $K$-vector space of all functions 
$(K^U)^{\times}\to K^{\times}$, so the morphism $K\langle G/U\rangle\to\prod_{(K^U)^{\times}}K$, 
given by $\sum_gb_g[g]\mapsto(\sum_gb_gf^g)_{f\in(K^U)^{\times}}$, 
is injective. Then, for any non-zero element $\alpha\in K\langle G/U\rangle$, there exists 
an element $Q\in K^U$ such that the morphism $K\langle G/U\rangle\to K$, given by 
$\sum_gb_g[g]\mapsto\sum_gb_gQ^g$, does not vanish on $\alpha$. 
Then $\alpha$ generates a $K'$-semilinear subrepresentation $V$ 
admitting a non-zero morphism to $K$. If $V$ is irreducible then it is 
isomorphic to $K'$, so $V^G\neq 0$. In particular, $K\langle G/U\rangle^G\neq 0$, 
which can happen only if index of $U$ in $G$ is finite. 

If $U$ is of finite index in $G$ set $U'=\cap_{g\in G/U}gUg^{-1}$. This is a normal subgroup 
of finite index. Then $K\langle G/U'\rangle=K\otimes_{K^{U'}}K^{U'}\langle G/U'\rangle$ 
and $K^{U'}\langle G/U'\rangle\cong(K^{U'})^{[G:U']}$ is trivial by Speiser's version  
of Hilbert's theorem 90, so we get $K\langle G/U'\rangle\cong K^{[G:U']}$. 

(i) and (ii) follow from the explicit description 
of open subgroups in Lemma \ref{open-subgrps-descr}. \end{proof} 

\begin{lemma} \label{finite-index-split} Let $K$ be a field, $G$ be a group of 
automorphisms of the field $K$. Let $U\subseteq H\subseteq G$ be open subgroups of $G$. 
Then the natural right $K^H$-vector space structure on $K\langle G/H\rangle$, given by 
$[g]\cdot f=f^g\cdot[g]$, commutes with the natural left $K$-vector space structure. 
If index of $U$ in $H$ is finite then there is a natural isomorphism 
$K\langle G/H\rangle\otimes_{K^H}K^U\stackrel{\sim}{\longrightarrow}K\langle G/U\rangle$, 
$[g]\otimes f\mapsto\sum_{[\xi]\in G/U,~[\xi]\bmod H=[g]}f^{\xi}[\xi]$. \end{lemma} 
\begin{proof} The injectivity follows from Artin's independence of characters theorem.  
To check the surjectivity, it suffices to check the surjectivity of the restriction 
$K\otimes_{K^H}K^U\stackrel{\sim}{\longrightarrow}K\langle H/U\rangle$, but this is 
Lemma \ref{no-simple-submod}. \end{proof} 

\begin{lemma} \label{semilin-gener} Let $K$ be a field, $G$ be 
a group of automorphisms of the field $K$. Let $B$ be such a system of 
open subgroups of $G$ that any open subgroup contains a subgroup 
conjugated, for some $H\in B$, to an open subgroup of finite index 
in $H$. Then the objects $K\langle G/H\rangle$ for all $H\in B$ form a system 
of generators of the category of smooth $K$-semilinear 
representations of $G$. \end{lemma} 
\begin{proof} Let $V$ be a smooth semilinear representation of $G$. Then the 
stabilizer of any vector $v\in V$ is open, i.e., the stabilizer of some vector $v'$ 
in the $G$-orbit of $v$ admits a subgroup commensurable with some $H\in B$. 
The $K$-linear envelope of the (finite) $H$-orbit of $v'$ is a smooth $K$-semilinear 
representation of $H$, so it is trivial, i.e., $v'$ belongs to the $K$-linear 
envelope of the $K^H$-vector subspace fixed by $H$. As a consequence, there is 
a morphism from a finite cartesian power of $K\langle G/H\rangle$ to $V$, containing $v'$ 
(and therefore, containing $v$ as well) in the image. \end{proof} 

\begin{example} \label{Omega-binom} Let $K$ be a field endowed 
with a smooth faithful $\Sy_{\Psi}$-action. 
Let $S\subseteq\mathbb N$ be an infinite set of positive integers. Then (i) the assumptions of 
Lemma \ref{semilin-gener} hold if $B$ is the set of subgroups $\Sy_{\Psi,T}$ for a collection of 
subsets $T\subset\Psi$ with cardinality in $S$, (ii) $K\langle\binom{\Psi}{N}\rangle$ is isomorphic to 
$K\langle\Sy_{\Psi}/\Sy_{\Psi,T}\rangle$ for any $T$ of order $N$. 

Thus, the objects $K\langle\binom{\Psi}{N}\rangle$ for $N\in S$ form a system of generators 
of the category of smooth $K$-semilinear representations of $\Sy_{\Psi}$. One has 
$K\langle\binom{\Psi}{N}\rangle\cong\bigwedge_K^NK\langle\Psi\rangle\cong\Omega^N_{K|k}$, 
$[\{s_1,\dots,s_N\}]\leftrightarrow\prod_{1\le i<j\le N}(s_i-s_j)
[s_1]\wedge\dots\wedge[s_N]\leftrightarrow\prod_{1\le i<j\le N}(s_i-s_j)
ds_1\wedge\dots\wedge ds_N$, if $K=k(\Psi)$. \qed \end{example} 

\section{Structure of smooth semilinear representations of $\Sy_{\Psi}$} 
The following result will be used in the particular case of the trivial 
$G$-action on the $A$-module $V$ (i.e., $\chi\equiv id_V$), claiming the 
injectivity of the natural map $A\otimes_{A^G}V^G\to V$ (since $V_{id_V}=V^G$). 
\begin{lemma} \label{inject} Let $G$ be a group, $A$ be 
a division ring endowed with a 
$G$-action $G\to\Aut_{\mathrm{ring}}(A)$, $V$ be an $A\langle G\rangle$-module 
and $\chi:G\to\Aut_A(V)$ be a $G$-action on the $A$-module $V$. 

Set $V_{\chi}:=\{w\in V~|~\sigma w=\chi(\sigma)w
\text{ {\rm for all} $\sigma\in G$}\}$. 

Then $V_{\chi}$ is an $A^G$-module and the natural map 
$A\otimes_{A^G}V_{\chi}\to V$ is injective. \end{lemma} 
\begin{proof} This is well-known: Suppose that some elements 
$w_1,\dots,w_m\in V_{\chi}$ are $A^G$-linearly independent, 
but $A$-linearly dependent for a minimal $m\ge 2$. 
Then $w_1=\sum_{j=2}^m\lambda_jw_j$ for some $\lambda_j\in A^{\times}$. 

Applying $\sigma-\chi(\sigma)$ for each $\sigma\in G$ to both sides of the latter 
equality, we get $\sum_{j=2}^m(\lambda_j^{\sigma}-\lambda_j)\chi(\sigma)w_j=0$, 
and therefore, $\sum_{j=2}^m(\lambda_j^{\sigma}-\lambda_j)w_j=0$. 
By the minimality of $m$, one has $\lambda_j^{\sigma}-\lambda_j=0$ for 
each $\sigma\in G$, so $\lambda_j\in A^G$ for any $j$, contradicting to 
the $A^G$-linear independence of $w_1,\dots,w_m$. \end{proof}

\subsection{Growth estimates} \label{substructures} 
Let $G\subseteq\Sy_{\Psi}$ be a permutation group of a set $\Psi$. 

For a subset $S\subset\Psi$, (i) we denote by $G_S$ the pointwise stabilizer of 
the set $S$; (ii) we call the fixed set $\Psi^{G_S}$ the $G$-{\sl closure} of $S$. 
We say that a subset $S\subset\Psi$ is $G$-{\sl closed} if $S=\Psi^{G_S}$. 

Any intersection $\bigcap_iS_i$ of $G$-closed sets $S_i$ is $G$-closed: 
as $G_{S_i}\subseteq G_{\bigcap_jS_j}$, one has $G_{S_i}s=s$ for any 
$s\in\Psi^{G_{\bigcap_jS_j}}$, so $s\in\Psi^{G_{S_i}}=S_i$ for any $i$, and thus, 
$s\in\bigcap_iS_i$. This implies that the subgroup generated by $G_{S_i}$'s 
is dense in $G_{\bigcap_iS_i}$ (and coincides with $G_{\bigcap_iS_i}$ if 
at least one of $G_{S_i}$'s is open). 

The $G$-closed subsets of $\Psi$ form a small concrete category with the 
morphisms being all those embeddings that are induced by elements of $G$. 

For a finite $G$-closed subset $T\subset\Psi$, 
(hiding $G$ and $\Psi$ from notation) set $\Aut(T):=N_G(G_T)/G_T$. 

Assume that for any integer $N\ge 0$ the $G$-closed subsets of length $N$ 
form a non-empty $G$-orbit. For each integer $N\ge 0$ fix a $G$-closed subset 
$\Psi_N\subset\Psi$ of length $N$, i.e., $N$ is the minimal cardinality of 
the subsets $S\subset\Psi$ such that $\Psi_N$ is the $G$-closure of $S$. 

For a division ring endowed with a $G$-action and an $A\langle G\rangle$-module 
$M$ define a function $d_M:{\mathbb Z}_{\ge 0}\to{\mathbb Z}_{\ge 0}\coprod\{\infty\}$ by 
$d_M(N):=\dim_{A^{G_{\Psi_N}}}(M^{G_{\Psi_N}})$. 
\begin{lemma} \label{growth} Let $G$ be either $\Sy_{\Psi}$ (and then $q:=1$) or the group 
of automorphisms of an $\mathbb F_q$-vector space $\Psi$ fixing a subspace of finite dimension 
$v\ge 0$. Let $A$ be a division ring endowed with 
a $G$-action. If $0\neq M\subseteq A\langle G/G_{\Psi_n}\rangle$ for some $n\ge 0$ 
then $d_M$ grows as a $q$-polynomial of degree $n$: 
\[\frac{1}{d_n(n)}([N]_q-[n+m-1]_q)^n\le\frac{d_{m+n}(N)}{d_m(N)d_n(n)}
\le d_M(N)\le q^{vn}d_n(N)\le q^{vn}[N]_q^n\] 
for some $m\ge 0$, where $[s]_q:=\#\Psi_s$ and $d_n(N)$ is the number of embeddings 
$\Psi_n\hookrightarrow\Psi_N$ induced by elements of $G$, 
which is $([N]_q-[0]_q)\cdots([N]_q-[n-1]_q)$. \end{lemma} 
\begin{proof} As $M^{G_{\Psi_N}}\subseteq A\langle N_G(G_{\Psi_N})/(N_G(G_{\Psi_N})\cap G_{\Psi_n})\rangle$ 
and (by Lemma \ref{inject}) $A\otimes_{A^{G_{\Psi_N}}}M^{G_{\Psi_N}}\to M\subseteq A\langle G/G_{\Psi_n}\rangle$ 
is injective, there is a natural inclusion \[A\otimes_{A^{G_{\Psi_N}}}M^{G_{\Psi_N}}\hookrightarrow 
A\langle N_G(G_{\Psi_N})/(N_G(G_{\Psi_N})\cap G_{\Psi_n})\rangle=A\langle\Aut(\Psi_N)/\Aut(\Psi_N|\Psi_n)\rangle,\] 
if $n\le N$. (Here $\Aut(\Psi_N|\Psi_n)$ denotes the automorphisms of $\Psi_N$ identical on 
$\Psi_n$.) Then one has 
$d_M(N)\le\#(\Aut(\Psi_N)/\Aut(\Psi_N|\Psi_n))=q^{vn}d_n(N)$. 
The lower bound of $d_M(N)$ is given by the number of $G$-closed 
subsets in $\Psi_N$ with length-0 intersection with $\Psi_m$. 
Indeed, for any non-zero element $\alpha\in M\subseteq A\langle G/G_{\Psi_n}\rangle$ there exist 
an integer $m\ge 0$ and elements $\xi,\eta\in G$ such that $\xi\alpha$ is congruent 
to $\sum_{\sigma\in\Aut(\Psi_n)}b_{\sigma}\eta\sigma$ for some non-zero collection 
$\{b_{\sigma}\in A\}_{\sigma\in\Aut(\Psi_n)}$ modulo monomorphisms whose 
images have intersection of positive length with a fixed finite $\Psi_m$. \end{proof} 

Let $q$ be either 1 or a primary integer. Let $S$ be a plain set if $q=1$ 
and an $\mathbb F_q$-vector space if $q>1$. For each integer $s\ge 0$, we 
denote by $\binom{S}{s}_q$ the set of subobjects of $S$ ($G$-closed subsets of $\Psi$, 
if $S=\Psi$, where $G=\Sy_{\Psi}$ if $q=1$ and $G=\mathrm{GL}_{\mathbb F_q}(\Psi)$ if $q>1$) of length $s$. 
In other words, $\binom{S}{s}_1:=\binom{S}{s}$, while $\binom{S}{s}_q$ is 
the Grassmannian of the $s$-dimensional subspaces in $S$ if $q>1$. 
\begin{corollary} \label{def-binom} Let $G$ be either 
$\Sy_{\Psi}$ (and then $q:=1$) or the group of automorphisms of an $\mathbb F_q$-vector space $\Psi$ 
fixing a finite-dimensional subspace of $\Psi$. Let $A$ be a division 
ring endowed with a $G$-action. Let $\Xi$ be a finite subset in 
$\Hom_{A\langle G\rangle}(A\langle G/G_T\rangle,A\langle G/G_{T'}\rangle)$ 
for some finite $G$-closed $T'\subsetneqq T\subset\Psi$. Then 
\begin{enumerate}\item \label{any-is-essential} any non-zero $A\langle G\rangle$-submodule 
of $A\langle\binom{\Psi}{m}_q\rangle$ is essential for any integer 
$m\ge 0$;\footnote{Recall, that an injection $M\hookrightarrow N$ in an abelian category 
is called an {\sl essential} extension if any non-zero subobject of $N$ has a non-zero 
intersection with the image of $M$.} 
\item\label{level} there are no nonzero isomorphic 
$A\langle G\rangle$-submodules in $A\langle G/G_T\rangle$ and $A\langle G/G_{T'}\rangle$; 
\item \label{kernel-of-finite} the common kernel $V_{\Xi}$ of 
all elements of $\Xi$ is an essential 
$A\langle G\rangle$-submodule in $A\langle G/G_T\rangle$. \end{enumerate} \end{corollary} 
\begin{proof} (\ref{any-is-essential}) follows from 
the lower growth estimate of Lemma \ref{growth}. 

(\ref{level}) follows immediately from Lemma \ref{growth}. 

(\ref{kernel-of-finite}) Suppose that there exists a nonzero submodule 
$M\subseteq A\langle G/G_T\rangle$ such that $M\cap V_{\Xi}=0$. Then 
restriction of some $\xi\in\Xi$ to $M$ is nonzero. If $\xi|_M$ is not 
injective, replacing $M$ with $\ker\xi\cap M$, we can assume that 
$\xi|_M=0$. In other words, we can assume that restriction to $M$ of 
any $\xi\in\Xi$ is either injective or zero. In particular, restriction 
to $M$ of some $\xi\in\Xi$ is injective, i.e. $\xi$ embeds $M$ into 
$A\langle G/G_{T'}\rangle$, contradicting to (\ref{level}). \end{proof} 

\subsection{Smooth $\Sy_{\Psi}$-sets and $F_{\Psi}$-semilinear representation 
of $\Sy_{\Psi}$ as sheaves}  

Let $\mathrm{FinEmb}$ be the 
following category. Its objects are the finite sets. 
Its morphisms are opposite to the embeddings. 
For each objet $T\in\mathrm{FinEmb}$ denote by 
$\mathrm{FinEmb}_T$ the category of morphisms to $T$. (E.g., $\mathrm{FinEmb}_T$ 
is equivalent to $\mathrm{FinEmb}$, $S\mapsto S\smallsetminus T$, $S\mapsto S\coprod T$.) 
The category $\mathrm{FinEmb}$ admits products: product of a pair of objects $T_1$, $T_2$ 
of $\mathrm{FinEmb}_T$ is $T_1\sqcup_TT_2$. 

Consider $\mathrm{FinEmb}_T$ as a site, where any morphism is covering.

\begin{lemma} \label{iso-restr} Let $\Psi$ be an infinite set. Let $F$ be a field and $k$ be 
a subfield algebraically closed in $F$. To each sheaf of sets $\mathcal F$ on $\mathrm{FinEmb}$ 
we associate the $\Sy_{\Psi}$-set $\mathcal F(\Psi):=\mathop{\underrightarrow{\lim}}
\limits_{J\subset\Psi}\mathcal F(J)$, where $J$ runs over the finite subsets of $\Psi$. 
Let $\mathcal O$ be the sheaf of fields $S\mapsto F_S$. 

This gives rise to the following equivalences of categories: 
\begin{gather*} \nu_{\Psi}:\{\text{{\rm sheaves of sets on $\mathrm{FinEmb}$}}\}
\stackrel{\sim}{\longrightarrow}\{\text{{\rm smooth $\Sy_{\Psi}$-sets}}\};\\ 
\{\text{{\rm sheaves of $k$-vector spaces on $\mathrm{FinEmb}$}}\}
\stackrel{\sim}{\longrightarrow}\{\text{{\rm smooth representations of 
$\Sy_{\Psi}$ over $k$}}\};\\
\{\text{{\rm sheaves of $\mathcal O$-modules on $\mathrm{FinEmb}$}}\}
\stackrel{\sim}{\longrightarrow}\{\text{{\rm smooth $F_{\Psi}$-semilinear 
representations of $\Sy_{\Psi}$}}\}.\end{gather*} 
The functor $\nu_{\Psi}$ admits a quasi-inverse $\nu_{\Psi}^{-1}$ such that for any infinite 
subset $\Psi'\subseteq\Psi$ the functor $\nu_{\Psi'}\circ\nu_{\Psi}^{-1}$ is given by 
$M\mapsto M':=\mathop{\underrightarrow{\lim}}\limits_{J\subset\Psi'}M^{\Sy_{\Psi|J}}\subseteq 
M^{\Sy_{\Psi|\Psi'}}$, where $J$ runs over the finite subsets of $\Psi'$,\footnote{This does 
not lead to confusion in the cases $M=\Psi$, since $\Psi'=\mathop{\underrightarrow{\lim}}\limits_{J\subset\Psi'}J
=\mathop{\underrightarrow{\lim}}\limits_{J\subset\Psi'}\Psi^{\Sy_{\Psi|J}}$.} 
and gives rise to the following equivalences of categories: 
\begin{gather*} \{\text{{\rm smooth $\Sy_{\Psi}$-sets}}\}
\stackrel{\sim}{\longrightarrow}\{\text{{\rm smooth $\Sy_{\Psi'}$-sets}}\};\\ 
\{\text{{\rm smooth representations of $\Sy_{\Psi}$ over $k$}}\}
\stackrel{\sim}{\longrightarrow}\{\text{{\rm smooth representations of 
$\Sy_{\Psi'}$ over $k$}}\};\\
\left\{\begin{array}{c}\text{{\rm smooth $F_{\Psi}$-semilinear}}\\ 
\text{{\rm representations of $\Sy_{\Psi}$}}\end{array}\right\}\stackrel{\sim}{\longrightarrow}
\left\{\begin{array}{c}\text{{\rm smooth $F_{\Psi'}$-semilinear}}\\ 
\text{{\rm representations of $\Sy_{\Psi'}$}}\end{array}\right\}.\end{gather*} 
\end{lemma} 
\begin{proof} For any pair of sheaves $\mathcal F$ and $\mathcal G$, a map of sets 
$\alpha:\mathcal F(\Psi)\to\mathcal G(\Psi)$ and any embedding $\iota:S\hookrightarrow\Psi$ such 
that $\alpha(\mathcal F(\Psi)^{\Sy_{\Psi|\iota(S)}})\subseteq\mathcal G(\Psi)^{\Sy_{\Psi|\iota(S)}}$ 
there is a unique map $\alpha_S:\mathcal F(S)\to\mathcal G(S)$ making commutative the square 
\[\begin{array}{ccc}\mathcal F(S)&\stackrel{\alpha_S}{\longrightarrow}&\mathcal G(S)\\
\downarrow\lefteqn{\iota_{\mathcal F}}&&\downarrow\lefteqn{\iota_{\mathcal G}}\\ 
\mathcal F(\Psi)&\stackrel{\alpha}{\longrightarrow}&\mathcal G(\Psi).\end{array}\] 
If $\alpha$ is a morphism of $\Sy_{\Psi}$-sets then $\alpha_S$ is 
independent of $\iota$, since all embeddings $S\hookrightarrow\Psi$ form a single 
$\Sy_{\Psi}$-orbit. This gives gives rise to a natural bijection 
$\Hom_{\Sy_{\Psi}}(\mathcal F(\Psi),\mathcal G(\Psi))\stackrel{\sim}{\longrightarrow}
\Hom(\mathcal F,\mathcal G)$, the inverse map is given by restriction to finite subsets 
of $\Psi$. 

To construct a functor $\nu_{\Psi}^{-1}$ quasi-inverse to $\nu_{\Psi}$, for each finite 
set $S$ fix an embedding $\iota_S:S\hookrightarrow\Psi$, which is identical if $S$ 
is a subset of $\Psi$. Then to a smooth $\Sy_{\Psi}$-set $M$ we associate the presheaf 
$S\mapsto M^{\Sy_{\Psi|\iota_S(S)}}$ and to each embedding $j:S\hookrightarrow T$ we 
associate a unique map $M^{\Sy_{\Psi|\iota_S(S)}}\hookrightarrow M^{\Sy_{\Psi|\iota_T(T)}}$ 
induced by the element 
$\iota_Tj\iota_S^{-1}\in\Sy_{\Psi|\iota_T(T)}\backslash\Sy_{\Psi}/\Sy_{\Psi|\iota_S(S)}$. 
It follows from Lemma~\ref{strong-generation} that this presheaf is a sheaf. \end{proof}

\subsection{Local structure of smooth semilinear representations of $\Sy_{\Psi}$} 
\begin{proposition} \label{local-structure} Let $K$ be a field endowed with a faithful 
smooth $\Sy_{\Psi}$-action. Then for any smooth finitely generated $K\langle\Sy_{\Psi}\rangle$-module $V$ 
there is a finite subset $J\subset\Psi$ and an isomorphism of $K\langle\Sy_{\Psi|J}\rangle$-modules 
$\bigoplus_{s=0}^NK\langle\binom{\Psi\smallsetminus J}{s}\rangle^{\kappa_s}\stackrel{\sim}{\longrightarrow}V$ 
for some integer $N,\kappa_0,\dots,\kappa_N\ge 0$. \end{proposition} 
\begin{proof} By Lemma \ref{semilin-gener}, there is a surjection of $K\langle G\rangle$-modules 
$K\langle\binom{\Psi}{N}\rangle^m\oplus\bigoplus_{s=0}^{N-1}K\langle\binom{\Psi}{s}\rangle^{m_s}\to V$ 
for some $N\ge 0$ and $m_s\ge 0$. The proof proceeds by induction on $N$, the case $N=0$ being trivial. 

The induction step proceeds by induction on $m$, the case $m=0$ being the induction 
assumption of the induction on $N$. Let $\alpha:K\langle\binom{\Psi}{N}\rangle^m\to V$ and 
$\beta:\bigoplus_{s=0}^{N-1}K\langle\binom{\Psi}{s}\rangle^{m_s}\to V$ be two morphisms such that 
$\alpha+\beta:K\langle\binom{\Psi}{N}\rangle^m\oplus\bigoplus_{s=0}^{N-1}K\langle\binom{\Psi}{s}
\rangle^{m_s}\to V$ is surjective. Suppose that $\alpha$ is injective. Then, by Lemma \ref{growth}, 
the images of $\alpha$ and of $\beta$ have zero intersection. Therefore, 
$V\cong K\langle\binom{\Psi}{N}\rangle^m\oplus\mathrm{Im}(\beta)$, thus, concluding the induction step. 
Suppose now that $\alpha$ is not injective. Then $\alpha$ factots through a quotient 
$K\langle\binom{\Psi}{N}\rangle^m/\langle(\xi_1,\dots,\xi_m)\rangle$ for a non-zero collection 
$(\xi_1,\dots,\xi_m)$. Without loss of generality, we may assume that $\xi_1\neq 0$, 
so $\xi_1=\sum_{i=1}^ba_iI_i$ for some $I_i\subset\Psi$ of order $N$ and non-zero $a_i$. 
Set $J:=\bigcup_{i=1}^bI_i\smallsetminus I_1$. Then the inclusion 
$K\langle\binom{\Psi}{N}\rangle^{m-1}\hookrightarrow K\langle\binom{\Psi}{N}\rangle^m$ induces 
a surjection of $K\langle G_J\rangle$-modules $K\langle\binom{\Psi}{N}\rangle^{m-1}\oplus
\bigoplus_{\Lambda\subsetneqq J}K\langle\binom{\Psi\smallsetminus J}{\#\Lambda}\rangle\to 
K\langle\binom{\Psi}{N}\rangle^m/\langle(\xi_1,\dots,\xi_m)\rangle$ 
giving rise to a surjection of $K\langle G_J\rangle$-modules 
$K\langle\binom{\Psi}{N}\rangle^{m-1}\oplus\bigoplus_{s=0}^{N-1}
K\langle\binom{\Psi\smallsetminus J}{s}\rangle^{\binom{\#J}{s}+m_s}\to V$. \end{proof} 

\begin{remark} By Krull--Schmidt--Remak--Azumaya Theorem the integers 
$N,\kappa_0,\dots,\kappa_N\ge 0$ in Proposition are uniquely determined. 
Clearly, $N$ and $\kappa_N$ are independent of $J$. We call $N$ {\sl level} of $V$. 
It is easy to show 
that any non-zero submodule of $K\langle\binom{\Psi\smallsetminus S}{N}\rangle$ is of level $N$. \end{remark} 

\begin{corollary} Let $K$ be a field endowed with a smooth $\Sy_{\Psi}$-action. 
Then any smooth finitely generated $K\langle\Sy_{\Psi}\rangle$-module $V$ is 
admissible, i.e., $\dim_{K^U}V^U<\infty$ for any open subgroup $U\subseteq\Sy_{\Psi}$. \qed 
\end{corollary} 

\begin{proposition} \label{injectivity-K} Let $\Psi$ be a set, $F$ be a field and $k$ be a subfield 
algebraically closed in $F$, $K=F_{\Psi}$ be the field defined on p.\pageref{NotaF} endowed with 
the standard $\Sy_{\Psi}$-action. Assume that transcendence degree of the field extension $F|k$ is 
at most continuum. 
Then the smooth $K\langle\Sy_{\Psi}\rangle$-module $K$ is an injective 
object of the category of smooth $K$-semilinear representations of $\Sy_{\Psi}$. \end{proposition} 
\begin{proof} Let a smooth $K\langle\Sy_{\Psi}\rangle$-module $E$ be an essential 
extension of $K$. We are going to show that $E=K$, so we may assume that $E$ is cyclic. 
By Proposition \ref{local-structure}, there is a finite subset $J\subset\Psi$ and 
an isomorphism of $K\langle\Sy_{\Psi|J}\rangle$-modules 
$\bigoplus_{s=0}^NK\langle\binom{\Psi\smallsetminus J}{s}\rangle^{\kappa_s}\stackrel{\sim}{\longrightarrow}E$ 
for some integer $N,\kappa_0,\dots,\kappa_N\ge 0$. Let, in notation of Lemma~\ref{iso-restr}, 
$E':=\mathop{\underrightarrow{\lim}}\limits_IE^{\Sy_{\Psi|I}}$, where $I$ runs over finite 
subsets of $\Psi\smallsetminus J$, so $E'$ is a cyclic $K'\langle\Sy_{\Psi|J}\rangle$-submodule 
of $\bigoplus_{s=0}^NK\langle\binom{\Psi\smallsetminus J}{s}\rangle^{\kappa_s}$ 
which is an essential extension of $K'$. The natural projection defines a morphism of 
$K'\langle\Sy_{\Psi|J}\rangle$-modules $\pi:E'\to K^{\kappa_0}$ injective on $K'\subseteq E'$. 

To show that $E'=K'$, we have to construct a morphism $\lambda:E'':=\pi(E')\to K'$ 
identical on $K'$. A morphism $\lambda$ is constructed as composition of (i) any $K$-linear 
morphism $K^{\kappa_0}\to K$, which is $K'$-rational and identical on 
$K'\subseteq(E'')^{\Sy_{\Psi|J}}\subset(K^{\kappa_0})^{\Sy_{\Psi|J}}=(K')^{\kappa_0}$ 
with (ii) a morphism of $K'\langle\Sy_{\Psi|J}\rangle$-modules $\xi$ from the fraction field 
$K$ of $K'\otimes_kF_J$ to $K'$ identical on $K'$. We define $\xi$ as follows. Let 
$k_0\subseteq k$ be the prime subfield. Then the cardinality of $k_0((t))$ is continuum, so 
transcendence degree of $k_0((t))$ (and of $k((t))$ over $k$ as well) is continuum. This 
implies that we can send the elements of a chosen transcendence basis of $F_J|k$ to elements 
of $k((t))$ algebraically independent over $k$. 
By \cite{MacLane}, this extends to an embedding of the field $F_J$ into the field 
$\overline{k}((t^{\mathbb Q}))$ of Hahn power series (i.e., of formal expressions of 
the form $\sum_{s\in\mathbb Q}a_st^s$ with $a_s\in\overline{k}$ such that 
the set $S=\{s\in\mathbb Q~|~a_s\neq 0\}$ is bounded from below and the set 
$\{s\in\mathbb Q~|~s<r,~a_s\neq 0\}$ is finite for each real $r<\sup S$), so 
$K$ becomes a subfield of $(K'\otimes_k\overline{k})((t^{\mathbb Q}))$. Let $\xi:K\to K'$ 
be the constant term of the Hahn power series expression. \end{proof} 

\subsection{Proofs of Theorems \ref{Gabriel-spectrum}, 
\ref{smooth-simple-degree-0} and \ref{unip}} 
\begin{lemma} \label{simple-in-K} Let $\Psi$ be a set and $J\subset\Psi$ be a subset. 
Let $F$ be a field and $k$ be a subfield algebraically closed in $F$. Then any 
simple $F_{\Psi\smallsetminus J}\langle\Sy_{\Psi|J}\rangle$-submodule $M$ of $F_{\Psi}$ 
coincides with $aF_{\Psi\smallsetminus J}$ for some $a\in F_J^{\times}$. In particular, 
$M$ is isomorphic to $F_{\Psi\smallsetminus J}$. \end{lemma} 
\begin{proof} Let $Q\in F_{\Psi}^{\times}$ be a non-zero element of $M$, so $Q=\alpha/\beta$ 
is a ratio of a pair of elements $\alpha,\beta\in F_{\Psi\smallsetminus J}\otimes_k\bigotimes_{k,~i\in I}A_i$ 
for a finite subset $I\subseteq J$ and a finitely generated $k$-subalgebras $A_i$ of $F$. There is 
a finite field extension $k'|k$ and a collection of $k$-algebra homomorphisms $\varphi_i:A_i\to k'$ 
such that for the $k$-algebra homomorphism $\varphi:=id_{}\otimes\prod_{i\in I}\varphi_i:
F_{\Psi\smallsetminus J}\otimes_k\bigotimes_{k,~i\in I}A_i\to F_{\Psi\smallsetminus J}\otimes_kk'$ 
one has $\varphi(\alpha\beta)\neq 0$. Then $\varphi$ gives rise to a non-zero morphism of 
$F_{\Psi\smallsetminus J}\otimes_kk'\langle\Sy_{\Psi|J}\rangle$-modules 
$M\otimes_kk'\to F_{\Psi\smallsetminus J}\otimes_kk'$. As the 
$F_{\Psi\smallsetminus J}\langle\Sy_{\Psi|J}\rangle$-modules $M\otimes_kk'$ and 
$F_{\Psi\smallsetminus J}\otimes_kk'$ are isomorphic to (finite) direct sums of copies, 
respectively, of $M$ and of $F_{\Psi\smallsetminus J}$, we get $M\cong F_{\Psi\smallsetminus J}$. 
Let $a\in M^{\Sy_{\Psi|J}}\cong k$ be non-zero. Then $M=aF_{\Psi\smallsetminus J}$. \end{proof} 

\begin{theorem} \label{smooth-simple} Let $\Psi$ be a set, $F$ be a field 
and $k$ be a subfield algebraically closed in $F$. 

Assume that transcendence degree of the field extension $F|k$ is at most continuum. 

Let $K\subseteq F_{\Psi}$ be an $\Sy_{\Psi}$-invariant subfield. Then the object $F_{\Psi}$ is 
an injective cogenerator of the category of smooth $K$-semilinear representations of $\Sy_{\Psi}$. 
In particular, {\rm (i)} 
any smooth $K$-semilinear representation of $\Sy_{\Psi}$ can be embedded into a direct product of 
copies of $F_{\Psi}$; {\rm (ii)} any smooth $F_{\Psi}$-semilinear representation of $\Sy_{\Psi}$ 
of finite length is isomorphic to a direct sum of copies of $F_{\Psi}$. \end{theorem} 
\begin{proof} By Proposition~\ref{local-structure}, 
for any smooth simple $F_{\Psi}\langle\Sy_{\Psi}\rangle$-module $M$ there is a finite 
subset $J\subset\Psi$ and an isomorphism of $F_{\Psi}\langle\Sy_{\Psi|J}\rangle$-modules 
$\bigoplus_{s=0}^NF_{\Psi}\langle\binom{\Psi\smallsetminus J}{s}\rangle^{\kappa_s}
\stackrel{\sim}{\longrightarrow}M$ for some integer $N,\kappa_0,\dots,\kappa_N\ge 0$.
By Lemma~\ref{iso-restr}, the $F_{\Psi}\langle\Sy_{\Psi}\rangle$-module $M$ 
admits a simple $F_{\Psi\smallsetminus J}\langle\Sy_{\Psi|J}\rangle$-submodule $M'$. 
By Lemmas~\ref{no-simple-submod} and \ref{simple-in-K}, there are no simple 
$F_{\Psi\smallsetminus J}\langle\Sy_{\Psi|J}\rangle$-submodules in 
$F_{\Psi}\langle\binom{\Psi\smallsetminus J}{s}\rangle$ for $s>0$, so $M'$ is isomorphic to 
$F_{\Psi\smallsetminus J}$, again by Lemma \ref{simple-in-K}, and thus, 
$M$ is isomorphic to $F_{\Psi}$. 

We have to show that for any smooth simple $F_{\Psi}\langle\Sy_{\Psi}\rangle$-module 
$V$ and any non-zero $v\in V$ there is a morphism $V\to F_{\Psi}$ non-vanishing at $v$. 
The $F_{\Psi}\langle\Sy_{\Psi}\rangle$-submodule $\langle v\rangle$ of $V$ generated by 
$v$ admits a simple quotient, which is just shown to be isomorphic to $F_{\Psi}$, i.e., 
there is a non-zero morphism $\varphi:\langle v\rangle\to F_{\Psi}$. As $F_{\Psi}$ is 
injective (Proposition \ref{injectivity-K}), $\varphi$ extends to $V$. \end{proof} 

\begin{corollary} \label{all-irr-in-triv} Let $k$ be a field and $\Psi$ be an 
infinite set. Let $\Sy_{\Psi}$ be the group of all permutations of the set 
$\Psi$ acting naturally on the field $F_{\Psi}$. Let $K\subset F_{\Psi}$ be an 
$\Sy_{\Psi}$-invariant subfield over $k$. Then any smooth $K$-semilinear irreducible 
representation of $\Sy_{\Psi}$ can be embedded into $F_{\Psi}$. \end{corollary} 
\begin{proof} For any smooth simple $K\langle\Sy_{\Psi}\rangle$-module 
$V$ the $F_{\Psi}\langle\Sy_{\Psi}\rangle$-module $V\otimes_KF_{\Psi}$ admits 
a simple quotient isomorphic, by Theorem \ref{smooth-simple}, to $F_{\Psi}$. 
This means that $V$ can be embedded into $F_{\Psi}$. \end{proof} 

\begin{corollary} \label{some-injective} Let $k$ be a field and $\Psi$ be an infinite set. Let 
$\Sy_{\Psi}$ be the group of all permutations of the set $\Psi$ acting naturally on the field 
$k(\Psi)$. Then the smooth $k(\Psi)$-semilinear representation $k(\Psi)\langle\binom{\Psi}{s}\rangle$ 
of $\Sy_{\Psi}$ is indecomposable and injective for any integer $s\ge 0$. \end{corollary} 
\begin{proof} Let $K\subset k(\Psi)$ be the subfield generated over $k$ by squares of 
the elements of $\Psi$. By Theorem \ref{smooth-simple}, $k(\Psi)$ is an injective object 
of the category of smooth $K\langle\Sy_{\Psi}\rangle$-modules. On the other hand, 
there is an isomorphism $\bigoplus_{s\ge 0}K\langle\binom{\Psi}{s}\rangle
\stackrel{\sim}{\longrightarrow}k(\Psi)$, $[S]\mapsto\prod_{t\in S}t\cdot K$, 
so each $K\langle\binom{\Psi}{s}\rangle$ is isomorphic to a direct summand of 
the injective smooth $K$-semilinear representation $k(\Psi)$ of $\Sy_{\Psi}$. \end{proof} 

\begin{proof}[Proof of Theorem \ref{Gabriel-spectrum}] 
Recall that the points of the Gabriel spectrum $\mathrm{Zar}(C)$ of a Grothendieck 
category $C$ are isomorphism classes of indecomposable injectives. Base of opens consists of 
sets of the form $[F]:=\{E\in\mathrm{Zar}(C)~|~\Hom(F,E)=0\}$ as $F$ ranges over the finitely 
presented objects. As $[F]\cap[G]=[F\oplus G]$, these sets are closed under finite intersection, 
so an arbitrary open set will have the form $\bigcup_i[F_i]$ with some finitely presented $F_i$. 

By Corollary \ref{some-injective}, we only have to show that any smooth finitely generated 
$F_{\Psi}\langle\Sy_{\Psi}\rangle$-module $V$ can be embedded into a direct sum of 
$F_{\Psi}\langle\binom{\Psi}{s}\rangle$ for several integer $s\ge 0$. 

By Proposition \ref{local-structure}, there is a subset $\Psi'\subset\Psi$ with finite 
complement $J$ and an isomorphism of $F_{\Psi}\langle\Sy_{\Psi|J}\rangle$-modules 
$\bigoplus_{s=0}^NF_{\Psi}\langle\binom{\Psi'}{s}\rangle^{\kappa_s}
\stackrel{\sim}{\longrightarrow}V$ for some integer $N,\kappa_0,\dots,\kappa_N\ge 0$. 
In particular, $V':=\mathop{\underrightarrow{\lim}}\limits_{I\subset\Psi'}V^{\Sy_{\Psi|I}}$, 
where $I$ runs over the finite subsets of $\Psi'$, can be embedded into 
$\bigoplus_{s=0}^NF_{\Psi}\langle\binom{\Psi'}{s}\rangle^{\kappa_s}$. 

By Lemma \ref{iso-restr}, it suffices to show that the $F_{\Psi'}\langle\Sy_{\Psi|J}\rangle$-module 
$F_{\Psi}\langle\binom{\Psi'}{s}\rangle$ is isomorphic to a direct sum of modules 
$F_{\Psi'},F_{\Psi'}\langle\Psi'\rangle,F_{\Psi'}\langle\binom{\Psi'}{2}\rangle,\dots$ 

We proceed by induction on the order of $J$, the case of empty $J$ being trivial. 
Suppose that this is known in the case $s=0$. As $F_{\Psi}\langle\binom{\Psi'}{s}\rangle
=F_{\Psi}\otimes_{F_{\Psi'}}F_{\Psi'}\langle\binom{\Psi'}{s}\rangle$, we only have to 
check that $F_{\Psi'}\langle\binom{\Psi'}{n}\times\binom{\Psi'}{s}\rangle\cong 
\bigoplus_{j=0}^{n+s}F_{\Psi'}\langle\binom{\Psi'}{j}\rangle^{\oplus N_j}$. It 
is clear that $F_{\Psi'}\langle\binom{\Psi'}{n}\times\binom{\Psi'}{s}\rangle\cong 
\bigoplus_{j=0}^{\min(n,s)}F_{\Psi'}\langle\binom{\Psi'}{j,n-j,s-j}\rangle$, where 
$\binom{\Psi'}{j,n-j,s-j}$ denotes the triples of disjoint subsets of $\Psi'$ of 
orders $j,n-j,s-j$. By Lemma~\ref{finite-index-split}, $F_{\Psi'}\langle\binom{\Psi'}{j,n-j,s-j}\rangle$ 
is isomorphic to a direct sum of copies of $F_{\Psi'}\langle\binom{\Psi'}{n+s-j}\rangle$. 

For the induction step when $s=0$, fix some $t\in J$ and set $L:=F_{\Psi'\sqcup(J\setminus\{t\})}$. 
Then, according to partial fraction decomposition, $L(t)=\bigoplus_{n=0}^{\infty}L\cdot t^n\oplus
\bigoplus_{m=1}^{\infty}\bigoplus_O\bigoplus_{j=0}^{\deg O-1}t^j\bigoplus_{P\in O}L\cdot P(t)^{-m}$, 
where $O$ runs over the $\Sy_{\Psi'}$-orbits of (non-constant) irreducible monic polynomials over $L$. In 
other words, $L(t)$ is a direct sum of summands isomorphic to $L$ and to $L\langle\Sy_{\Psi'}/U_{P,m}\rangle$ 
for some open subgroups $U_{P,m}\subseteq\Sy_{\Psi'}$. Applying Lemmas \ref{open-subgrps-descr} 
and \ref{finite-index-split} completes the induction step. \end{proof} 

\begin{proof}[Proof of Theorem \ref{smooth-simple-degree-0}] 
By Theorem \ref{smooth-simple}, $k(\Psi)$ is an injective cogenerator of the category of 
smooth $V_0\langle\Sy_{\Psi}\rangle$-modules. To show that the subobjects $V_d\subset k(\Psi)$ 
form a system of injective cogenerators, it suffices to verify that they are direct summands 
of $k(\Psi)$ and that $k(\Psi)$ embeds into $\prod_{d\in\mathbb Z}V_d$. 

There is a unique discrete valuation $v:k(\Psi)^{\times}\to\mathbb Z$ trivial on $V_0^{\times}$ 
and such that $v(x)=-1$ for some (equivalently, any) $x\in\Psi$. The valuation $v$ is 
$\Sy_{\Psi}$-invariant and completion of $k(\Psi)$ with respect to $v$ is isomorphic to the field 
of Laurent series $V_0((x^{-1}))=\mathop{\underrightarrow{\lim}}\limits_n\prod_{d\le n}V_0\cdot x^d
=\mathop{\underrightarrow{\lim}}\limits_n\prod_{d\le n}V_d\subset\prod_{d\in\mathbb Z}V_d$, 
so for each $d\in\mathbb Z$ there is a morphism of $V_0\langle\Sy_{\Psi}\rangle$-modules 
$k(\Psi)\to V_d$ splitting the inclusion $V_d\subset k(\Psi)$. This implies that all $V_d$ 
are direct summands of $k(\Psi)$, and thus, they are injective. \end{proof} 

\begin{remark} It follows from the above that the maximal semisimple 
$V_0\langle\Sy_{\Psi}\rangle$-submodule in $k(\Psi)$ coincides with 
$\bigoplus_{d\in\mathbb Z}V_d\subset k(\Psi)$. \end{remark} 

\begin{proof}[Proof of Theorem \ref{unip}] By Theorem \ref{smooth-simple}, $k(\Psi)$ is 
an injective cogenerator of the category of smooth $K\langle\Sy_{\Psi}\rangle$-modules. 
One has $k(\Psi)=K[x]\oplus\bigoplus_R\bigoplus_{m\ge 1}V^{(m)}_R$, 
where $R$ runs over the $\Sy_{\Psi}$-orbits of non-constant irreducible monic 
polynomials in $K[x]$ and $V^{(m)}_R$ is the $K$-linear envelope of $P(x)/Q^m$ 
for all $Q\in R$ and $P\in K[x]$ with $\deg P<\deg Q$. As $k(\Psi)$ is injective, 
its direct summand $K[x]$ is also injective, as well as $V^{(m)}_R$ for all $R$ and $m$. 

Each $V^{(m)}_R$ is filtered by $V^{(j,m)}_R$, 
$0\le j<\deg R$, where $V^{(j,m)}_R$ is the $K$-linear envelope of $P(x)/Q^m$ for all 
$Q\in R$ and $P\in K[x]$ with $\deg P\le j$. Clearly, these decomposition and filtration 
are independent of $x$. It suffices to show that the only simple 
$K\langle\Sy_{\Psi}\rangle$-submodule of $K[x]$ is $K$ and there are no simple 
$K\langle\Sy_{\Psi}\rangle$-submodules in $V^{(j,m)}_R$ for any $R$, $m$ and $j$. 

Suppose first that $V\subset K[x]$. Let $Q\in V$ be a (non-zero) monic polynomial in $x$ of 
minimal degree. Then $V$ contains $Q-\sigma Q$ for any $\sigma\in\Sy_{\Psi}$. If $\sigma Q\neq Q$ 
for some $\sigma\in\Sy_{\Psi}$ then $Q-\sigma Q\neq 0$ and $\deg(Q-\sigma Q)<\deg Q$, 
contradicting our assumption, so $\sigma Q=Q$ for any $\sigma\in\Sy_{\Psi}$, i.e., $Q\in k$. 

Suppose now that $V\subset V^{(j,m)}_R$. One has isomorphisms 
\[x^j\cdot:V^{(0,m)}_R\stackrel{\sim}{\longrightarrow}V^{(j,m)}_R/V^{(j-1,m)}_R\] 
for all $0\le j<\deg R$, so it suffices to check that $V^{(0,m)}_R$ admits no simple 
$K\langle\Sy_{\Psi}\rangle$-submodules. Fix some $Q\in R$. Then the morphism 
$K\langle\Sy_{\Psi}/\mathrm{Stab}_Q\rangle\to V^{(0,m)}_R$, $[g]\mapsto(gQ)^{-m}$, is an isomorphism. 
By Lemma \ref{no-simple-submod}, there are no simple submodules in $K\langle\Sy_{\Psi}/\mathrm{Stab}_Q\rangle$. 

Thus, any smooth $K\langle\Sy_{\Psi}\rangle$-module $V$ of finite length is a 
finite-dimensional $K$-vector space. Set $N:=\dim_KV$. By Theorem~\ref{smooth-simple}, 
the $\Sy_{\Psi}$-action on $V$ in a fixed basis is given by the 1-cocycle 
$f_{\sigma}=\Phi(I)\Phi(\sigma I)^{-1}$ for some finite $I\subset\Psi$ and 
some $\Phi(X)\in\mathrm{GL}_Nk(I)$. As $f_{\sigma}\in\mathrm{GL}_NK$, one has 
$\Phi(T_{\lambda}I)\Phi(T_{\lambda}\sigma I)^{-1}=\Phi(I)\Phi(\sigma I)^{-1}$ 
for any $\lambda\in\overline{k}$ and any $\sigma\in\Sy_{\Psi}$, where 
$T_{\lambda}x=x+\lambda$ for any $x\in\Psi\subset k(\Psi)$, and therefore, 
$\Phi(I)^{-1}\Phi(T_{\lambda}I)\in(\mathrm{GL}_Nk(I))^{\Sy_{\Psi}}=\mathrm{GL}_Nk$. 
Then $\lambda\mapsto\Phi(I)^{-1}\Phi(T_{\lambda}I)$ gives rise to a homomorphism 
of algebraic $k$-groups $\mathbb G_{a,k}\to\mathrm{GL}_{N,k}$. 
Changing the basis, we may assume that $\Phi(I)^{-1}\Phi(T_{\lambda}I)$ is block-diagonal 
with unipotent blocks corresponding to indecomposable direct summands of $V$. 
For any integer $N\ge 1$ the unique isomorphism class of smooth $K$-semilinear 
indecomposable representations of $\Sy_{\Psi}$ of length $N$ is presented by 
$\bigoplus_{j=0}^{N-1}x^jK\subset k(\Psi)$ for any $x\in\Psi$. 

To show that the object $K[x]$ is a cogenerator, it suffices to verify that for any 
smooth $K\langle\Sy_{\Psi}\rangle$-module $V$ and any non-zero $v\in V$ there is 
a morphism $V\to K[x]$ non-vanishing at $v$. The $K\langle\Sy_{\Psi}\rangle$-submodule 
in $V$ generated by $v$ admits a simple quotient, which is isomorphic, as we know, 
to $K$. So this submodule admits a morphism to $K[x]$ non-vanishing at $v$. 
By injectivity of $K[x]$, this morphism extends to $V\to K[x]$. \end{proof} 

\begin{corollary} In the setting of Theorem \ref{unip}, the smooth $K$-semilinear 
representations $K[x]$ and $K\langle\binom{\Psi}{s}\rangle$ of $\Sy_{\Psi}$ are 
indecomposable and injective for any integer $s\ge 1$. \end{corollary} 
\begin{proof} It is shown in the proof Theorem \ref{unip} that $V^{(m)}_R$ 
is injective for all $R$ and $m$. Then $K\langle\binom{\Psi}{s}\rangle$ 
is isomorphic to a direct summand of an appropriate $V^{(m)}_R$. \end{proof} 

\subsection{Noetherian properties of smooth semilinear representations of $\Sy_{\Psi}$} 
\begin{lemma} \label{indecomp} Let $G$ be a group acting on a field $K$. 
Let $U$ be a subgroup of $G$ such that $(G/U)^U=\{[U]\}$ {\rm (}i.e., 
$\{g\in G~|~gU\subseteq Ug\}=U${\rm )} and $[U:U\cap(gUg^{-1})]=\infty$, 
unless $g\in U$. Then $\End_{K\langle G\rangle}(K\langle G/U\rangle)=K^U$ is a field, 
so $K\langle G/U\rangle$ is indecomposable. \end{lemma} 
\begin{proof} Indeed, $\End_{K\langle G\rangle}(K\langle G/U\rangle)=(K\langle G/U\rangle)^U=K^U
\oplus(K\langle(G\smallsetminus U)/U\rangle)^U$. As $U(gUg^{-1})$ 
consists of $[U:U\cap(gUg^{-1})]$ classes in $G/(gUg^{-1})$, 
we see that $(K\langle(G\smallsetminus U)/U\rangle)^U=0$. \end{proof} 

{\sc Examples.} 1. Let $\Psi$ be an infinite set, possibly endowed with 
a structure of a projective space. Let $G$ be the group of automorphisms 
of $\Psi$, respecting the structure, if any. Let $J$ be the $G$-closure 
of a finite subset in $\Psi$, i.e., a finite subset or a finite-dimensional 
subspace. Let $U$ be the stabilizer of $J$ in $G$. Then $G/U$ is identified 
with the set of all $G$-closed subsets in $\Psi$ of the same length as $J$. 

2. By Lemma \ref{indecomp}, $K\langle G/U\rangle$ is indecomposable in the following 
examples: \begin{enumerate} \item $G$ is the group of projective 
automorphisms of an infinite projective space $\Psi$ (i.e., either 
$\Psi$ is infinite-dimensional, or $\Psi$ is defined over an infinite 
field), $U$ is the setwise stabilizer in $G$ of a finite-dimensional 
subspace $J\subseteq\Psi$. Then $G/U$ is identified with the 
Grassmannian of all subspaces in $\Psi$ of the same dimension as $J$.  
\item $G$ is the group of permutations of an infinite set $\Psi$, 
$U$ is the stabilizer in $G$ of a finite subset $J\subset\Psi$. 
Then $G/U$ is identified with the set $\binom{\Psi}{\#J}$ of all 
subsets in $\Psi$ of order $\#J$. 
\item $G$ is the automorphism group of an algebraically closed extension 
$F$ of a field $k$, $U$ is the stabilizer in $G$ of an algebraically 
closed subextension $L|k$ of finite transcendence degree. Then $G/U$ is 
identified with the set of all subextensions in $F|k$ isomorphic to $L|k$. 
\end{enumerate} 

\begin{lemma} \label{restriction} Let $G$ be a permutation group of a set, $A$ 
be an associative ring endowed with a smooth $G$-action and $U\subseteq G$ 
be an open subgroup. Then any smooth $A\langle G\rangle$-module is also 
smooth when considered as an $A\langle U\rangle$-module. Suppose that 
the set $U\backslash G/U'$ is finite for any open subgroup 
$U'\subseteq G$. Then the restriction of any smooth finitely generated 
$A\langle G\rangle$-module to $A\langle U\rangle$ 
is a finitely generated $A\langle U\rangle$-module. \end{lemma} 
\begin{proof} The $A\langle G\rangle$-modules $A\langle G/U'\rangle$ for all open 
subgroups $U'$ of $G$ form a generating family of the category of 
smooth $A\langle G\rangle$-modules. It suffices, thus, to check 
that $A\langle G/U'\rangle$ is a finitely generated $A\langle U\rangle$-module for all open 
subgroups $U'$ of $G$. Choose representatives $\alpha_i\in G/U'$ of 
the elements of $U\backslash G/U'$. Then $G/U'=\coprod_iU\alpha_i$, 
so $A\langle G/U'\rangle\cong\bigoplus_iA\langle U/(U\cap\alpha_iU'\alpha_i^{-1})\rangle$ 
is a finitely generated $A\langle U\rangle$-module. \end{proof} 

{\sc Examples.} 1. The finiteness assumption of Lemma \ref{restriction} 
is valid for any open subgroup $G$ of $\Sy_{\Psi}$ or of the automorphism group of 
an infinite-dimensional vector space over a finite field, 
as well as for any compact group $G$. 

2. The restriction functor splits the indecomposable generators 
into finite direct sums of indecomposable generators via canonical 
isomorphisms of $A\langle G_J\rangle$-modules $A\langle\binom{\Psi}{t}_q\rangle
=\bigoplus_{\Lambda\subseteq J}M_{\Lambda}$, where $M_{\Lambda}$ is 
the free $A$-module on 
the set of all subobjects of $\Psi$ of length 
$t$ and meeting $J$ along $\Lambda$. 

\vspace{4mm} 

In the following result, our principal examples of the ring $A$ will be 
division rings endowed with an $\Sy_{\Psi}$-action, though localization of 
$\mathbb Z[x~|~x\in\Psi]$ at all non-constant indecomposable polynomials 
gives one more example. 
\begin{proposition} \label{noetherian-generators} Let $A$ be an associative 
left noetherian ring endowed with an arbitrary $\Sy_{\Psi}$-action. 
Then the left $A\langle U\rangle$-module $A\langle\Psi^s\rangle$ is noetherian for 
any integer $s\ge 0$ and any open subgroup $U\subseteq\Sy_{\Psi}$. If 
the $\Sy_{\Psi}$-action on $A$ is smooth then any smooth finitely generated 
$A\langle\Sy_{\Psi}\rangle$-module is noetherian. \end{proposition} 
\begin{proof} We need to show that any $A\langle U\rangle$-submodule 
$M\subset A\langle\Psi^s\rangle$ is finitely generated for all $U=\Sy_{\Psi|S}$ 
with finite $S\subset\Psi$. We proceed by induction on $s\ge 0$, 
the case $s=0$ being trivial. Assume that $s>0$ and the 
$A\langle U\rangle$-modules $A\langle\Psi^j\rangle$ are noetherian for all $j<s$. 
Fix a subset $I_0\subset\Psi\smallsetminus S$ of order $s$. 

Let $M_0$ be the image of $M$ under the $A$-linear projector 
$\pi_0:A\langle\Psi^s\rangle\to A\langle I_0^s\rangle\subset A\langle\Psi^s\rangle$ 
omitting all $s$-tuples containing elements other than those of $I_0$. As $A$ 
is noetherian and $I_0^s$ is finite, the $A$-module $M_0$ is finitely generated. 
Let the $A$-module $M_0$ be generated by the images of some elements 
$\alpha_1,\dots,\alpha_N\in M\subseteq A\langle\Psi^s\rangle$. Then 
$\alpha_1,\dots,\alpha_N$ belong to the $A$-submodule $A\langle I^s\rangle$ of 
$A\langle\Psi^s\rangle$ for some finite subset $I\subset\Psi$. 

Let $J\subset I\cup S$ be the complement to $I_0$. For each pair $\gamma=(j,x)$, where 
$1\le j\le s$ and $x\in J$, set $\Psi^s_{\gamma}:=\{(x_1,\dots,x_s)\in\Psi^s~|~x_j=x\}$. 
This is a smooth $\Sy_{\Psi|J}$-set. Then the set $\Psi^s$ is the union of the 
$\Sy_{\Psi|J}$-orbit consisting of $s$-tuples of pairwise distinct elements of 
$\Psi\smallsetminus J$ and of a finite union of $\Sy_{\Psi|J}$-orbits embeddable into 
$\Psi^{s-1}$: $\bigcup_{\gamma}\Psi^s_{\gamma}\cup\bigcup_{1\le i<j\le s}\Delta_{ij}$, 
where $\Delta_{ij}:=\{(x_1,\dots,x_s)\in\Psi^s~|~x_i=x_j\}$ are diagonals. 

As (i) $M_0\subseteq\sum_{j=1}^NA\alpha_j+
\sum_{\gamma\in\{1,\dots,s\}\times J}A\langle\Psi^s_{\gamma}\rangle$, (ii) 
$g(M_0)\subset A\langle\Psi^s\rangle$ is determined by $g(I_0)$, (iii) for any 
$g\in U$ such that $g(I_0)\cap J=\varnothing$ there exists 
$g'\in U_J$ with $g(I_0)=g'(I_0)$ ($U_J$ acts transitively 
on the $s$-configurations in $\Psi\smallsetminus J$), 
one has inclusions of $A\langle U_J\rangle$-modules 
\[\sum_{j=1}^NA\langle U\rangle\alpha_j\subseteq M\subseteq\sum_{g\in U}
g(M_0)\subseteq\sum_{g\in U_J}g(M_0)
+\sum_{\gamma\in\{1,\dots,s\}\times J}A\langle\Psi^s_{\gamma}\rangle.\] 
On the other hand, $g(M_0)\subseteq g(\sum_{j=1}^NA\alpha_j)+
\sum_{\gamma\in\{1,\dots,s\}\times J}A\langle\Psi^s_{\gamma}\rangle$ 
for $g\in U_J$, 
and therefore, the $A\langle U_J\rangle$-module 
$M/\sum_{j=1}^NA\langle U\rangle\alpha_j$ becomes a subquotient of the 
noetherian, by the induction assumption, $A\langle U_J\rangle$-module 
$\sum_{\gamma\in\{1,\dots,s\}\times J}A\langle\Psi^s_{\gamma}\rangle$, so 
the $A\langle U_J\rangle$-module $M/\sum_{j=1}^NA\langle U\rangle\alpha_j$ 
is finitely generated, and thus, $M$ is finitely generated as well. \end{proof} 

As a corollary we get the following 
\begin{theorem} \label{locally-noetherian} Let $A$ be a left noetherian associative 
ring endowed with a smooth $\Sy_{\Psi}$-action. Then any smooth finitely generated 
left $A\langle\Sy_{\Psi}\rangle$-module $W$ is noetherian if considered as a left 
$A\langle U\rangle$-module for any open subgroup $U\subseteq\Sy_{\Psi}$. \end{theorem} 
\begin{proof} The module $W$ is a quotient of a finite direct 
sum of $A\langle\Psi^m\rangle$ for some integer $m\ge 0$, while $A\langle\Psi^m\rangle$ 
are noetherian by Proposition \ref{noetherian-generators}. \end{proof} 

In particular, the category of smooth $A\langle\Sy_{\Psi}\rangle$-modules 
is locally noetherian, i.e., any smooth finitely generated left 
$A\langle\Sy_{\Psi}\rangle$-module is noetherian. 

\section{Relation between representations of automorphism groups of universal 
domains and of symmetric groups: some examples} 
\subsection{0-cycles and representations} \label{prehist} We keep notations of \S\S\ref{motiv} 
and \ref{substructures}, so $C$ is an algebraically closed extension of infinite transcendence degree 
of an algebraically closed field $k$ of characteristic 0, $\Psi\subset C$ is an infinite transcendence 
base of $C$ over $k$, $G$ is the group of all automorphisms of the field $C$ leaving $k$ fixed and 
$G_{\Psi}$ is the subgroup of $G$ consisting of elements identical on $\Psi$ (or equivalently, on 
$k(\Psi)$). 

Denote by $\mathcal I_G(k)$ the category of smooth $k$-linear representations $V$ of $G$ such that 
$V^{G_L}=V^{G_{L'}}$ for any purely transcendental field subextension $L'|L$ in $C|k$. 

There are some reasons to expect that the following holds 
(\cite[Conjecture on p.513]{pgl}). 
\begin{conjecture} \label{main-conj} Any simple object of $\mathcal I_G(k)$ can be embedded 
into the tensor algebra $\bigotimes_C^{\bullet}\Omega^1_{C|k}$. \end{conjecture} 

This conjecture has consequences for the Chow groups $CH_0(-)^0$ of 0-cycles of degree 0. 
\begin{corollary}[\cite{pgl}, Corollary 7.9; \cite{MMJ}, Corollary 3.2] Assume that 
Conjecture~\ref{main-conj} holds and a rational map $f:Y\dasharrow X$ of smooth proper 
$k$-varieties induces injections $\Gamma(X,\Omega^q_{X|k})\to\Gamma(Y,\Omega^q_{Y|k})$ 
for all $q\ge 0$. Then $f$ induces a surjection $CH_0(Y)\to CH_0(X)$. 

If $\Gamma(X,\Omega^q_{X|k})=0$ for all $q\ge 2$ then the Albanese map induces an isomorphism 
between $CH_0(X)^0$ and the group of $k$-points of the Albanese variety of $X$. (The converse 
for `big' fields $k$, due to Mumford, is well-known.) \end{corollary}

\begin{example} Let $r\ge 1$ be an integer and $X$ be a smooth proper $k$-variety with 
$\Gamma(X,\Omega^j_{X|k})=0$ for all $r<j\le\dim X$. Let $Y$ be a sufficiently general 
$r$-dimensional plane section of $X$. Then the inclusion $Y\to X$ induces an injection 
$\Gamma(X,\Omega^{\bullet}_{X|k})\to\Gamma(Y,\Omega^{\bullet}_{Y|k})$. \end{example}

\begin{remark} Though the direct summands of $\bigotimes_C^{\bullet}\Omega^1_{C|k}$ 
are the only known explicit irreducible smooth $C$-semilinear representations of 
$G$, there is a continuum of others, at least if $C$ is countable, cf. \cite[Prop.3.5.2]{RMS}. 
However, Conjecture \ref{main-conj} relates the `interesting' irreducible 
smooth semilinear representations of $G$ to K\"ahler differentials. \end{remark}

\subsection{The functor $H^0(G_{\Psi},-)$} 
As it follows from Proposition~\ref{Satz1}, the functor $H^0(G_{\Psi},-)$ from the category 
of smooth $C\langle G\rangle$-modules to the category of smooth 
$k(\Psi)\langle\Sy_{\Psi}\rangle$-modules is a faithful and exact. 
However, it is not full: $\Omega^1_{C|k}$ and $\mathrm{Sym}^2_C\Omega^1_{C|k}$ are distinct simple 
smooth $C\langle G\rangle$-modules, while $H^0(G_{\Psi},\Omega^1_{C|k})=\Omega^1_{k(\Psi)|k}$ 
and $H^0(G_{\Psi},\mathrm{Sym}^2_C\Omega^1_{C|k})\cong\Omega^1_{k(\Psi)|k}\oplus\Omega^2_{k(\Psi)|k}$. 

Set $\Upsilon:=\{g\in G~|~g(\Psi)=\Psi\}$. 
For any smooth $k(\Psi)$-semilinear representation $V$ of $\Sy_{\Psi}$, $V\otimes_{k(\Psi)}C$ is 
naturally a smooth $C$-semilinear representation of $\Upsilon$. Assume that $W$ is $V\otimes_{k(\Psi)}C$ 
endowed with a smooth $C$-semilinear $G$-action extending the $\Upsilon$-action. It follows from 
\cite[Proposition 2.5]{max} that any open subgroup of $G$ containing $\Upsilon_T$ 
contains the subgroup $G_T$. 

1. Let $W_0:=V^{\Sy_{\Psi}}\otimes_kC$ correspond to the sum of all copies of $k(\Psi)$ in $V$. 
Then $W_0=W^{\Upsilon}\otimes_kC$. On the other hand, any vector of $W^{\Upsilon}$ is 
fixed by an open subgroup of $G$ containing $\Upsilon$, i.e., $W^{\Upsilon}=W^G$, 
and thus, $W_0=W^G\otimes_kC$ is a direct sum of copies of $C$. 

2. If $H^0(G_{\Psi},W)\cong\Omega^i_{k(\Psi)|k}$ then $W\cong\Omega^i_{C|k}$, at least if $k$ 
is the field of algebraic numbers. 
\begin{proof} Indeed, $W$ is admissible and irreducible, so we can apply \cite{adm}. \end{proof} 

\begin{proposition} For any {\sl admissible} $C\langle G\rangle$-module $V$ (i.e., $\dim_{C^U}V^U<\infty$ 
for any open subgroup $U\subset G$) one has $H^0(G_{\Psi},V)\cong
\bigoplus_{i=0}^{\infty}(\Omega^i_{k(\Psi)|k})^{\oplus m_i}$ as 
$k(\Psi)\langle\Sy_{\Psi}\rangle$-modules for some integer $m_i\ge 0$. \end{proposition} 
\begin{proof} 
By Corollary \ref{some-injective}, the objects 
$k(\Psi)\langle\binom{\Psi}{i}\rangle$ are injective, so using Theorem \ref{locally-noetherian} and 
identifications $k(\Psi)\langle\binom{\Psi}{i}\rangle\stackrel{\sim}{\longrightarrow}\Omega^i_{k(\Psi)|k}$ 
of Example \ref{Omega-binom}, 
we may assume that $V$ is cyclic. For any finite $T\subset\Psi$ consider the 
$k(T)$-semilinear representation $V^{G_T}$ of the group of $k$-linear automorphisms of 
the $k$-linear span of $T$. As it follows from \cite{adm}, $V^{G_T}$ admits a filtration whose 
quotients are direct summands of $k(T)$-tensor powers (Schur functors) of $\Omega^1_{k(T)|k}$. 
Moreover, for $T\subset T'$ one has $V^{G_T}\otimes_{k(T)}k(T')\subseteq V^{G_{T'}}$ and these 
filtrations are compatible, thus, giving rise to an ascending filtration on $V^{G_{\Psi}}$ whose 
quotients are direct summands of $k(\Psi)$-tensor powers of $\Omega^1_{k(\Psi)|k}$, so it remains 
to show that $k(\Psi)\langle\Psi\rangle^{\otimes_{k(\Psi)}^N}$ is isomorphic to 
$\bigoplus_{i=0}^Nk(\Psi)\langle\binom{\Psi}{i}\rangle^{\oplus a_{i,N}}$ for any 
integer $N\ge 0$, where 
$t^N=:\sum_{i=0}^Na_{i,N}\binom{t}{i}\in\mathbb Z[t]$. 

We proceed by induction on $N$, the cases $N\le 1$ being trivial. For the induction step it suffices 
to construct a bijective morphism \[\alpha:k(\Psi)\langle\binom{\Psi}{N}\rangle^{\oplus N}\oplus 
k(\Psi)\langle\binom{\Psi}{N+1}\rangle^{\oplus(N+1)}\longrightarrow k(\Psi)\langle\binom{\Psi}{N}
\rangle\otimes_{k(\Psi)}k(\Psi)\langle\Psi\rangle=k(\Psi)\langle\binom{\Psi}{N}\times\Psi\rangle.\] 
Denote by $\sigma_s$ the elementary symmetric polynomials and set 
$\alpha([S]_s):=\sum_{t\in S}\sigma_s(S\setminus\{t\})[S,t]$, $0\le s<N$, 
and $\alpha([T]_s):=\sum_{t\in T}\sigma_s(T\setminus\{t\})[T\setminus\{t\},t]$, $0\le s\le N$. 
As the elementary symmetric polynomials are algebraically independent, $\alpha$ is injective. 
The surjectivity follows from the coincidence of $k(T)$-dimensions of $k(T)\langle\binom{T}{N}
\rangle^{\oplus N}\oplus k(T)\langle\binom{T}{N+1}\rangle^{\oplus(N+1)}$ and of $k(T)\langle\binom{T}{N}
\rangle\otimes_{k(T)}k(T)\langle T\rangle$ for all finite subsets $T\subset\Psi$. \end{proof} 

\begin{proposition} Let $W\in\mathcal I_G(k)$. 
Then $H^0(G_{\Psi},W\otimes_kC)$ is injective. \end{proposition} 
\begin{proof} Let $\Pi$ be the set of isomorphism classes of smooth irreducible representations of 
$G_{\Psi}$. For any $\overline{\rho}\in\Pi$ the subgroup $\ker\overline{\rho}\subset G_{\Psi}$ is 
open, so the subfield $C^{\ker\overline{\rho}}$ is a finite extension of $k(\Psi)$, and thus, it is 
a purely transcendental extension of a subfield $L_{\overline{\rho}}$ finitely generated over $k$. 

Denote by $W_{\overline{\rho}}=\rho\otimes_k\Hom_{G_{\Psi}}(\rho,W)$ the 
$\overline{\rho}$-isotypical part, where $\rho$ is a representation in $\overline{\rho}$, 
so $W=\bigoplus_{\overline{\rho}\in\Pi}W_{\overline{\rho}}$. Then $H^0(G_{\Psi},W\otimes_kC)=
\bigoplus_{\overline{\rho}\in\Pi}H^0(G_{\Psi},W_{\overline{\rho}}\otimes_kC_{\overline{\rho}^{\vee}})
=\bigoplus_OV_O$, where $O$ runs over the $\Sy_{\Psi}$-orbits in $\Pi$ and 
$V_O=\bigoplus_{\overline{\rho}\in O}H^0(G_{\Psi},W_{\overline{\rho}}\otimes_kC_{\overline{\rho}^{\vee}})$. 
For any $\overline{\rho}\in\Pi$ and $g\in G$ one has $g(W_{\overline{\rho}})\subseteq 
g(W^{\ker\overline{\rho}})\subseteq g(W^{G_{L_{\overline{\rho}}}})=W^{G_{g(L_{\overline{\rho}})}}$, 
so the pointwise stabilizer $\mathrm{Stab}_{\overline{\rho}}$ of $W_{\overline{\rho}}$ is open. 

Denote by $\mathrm{St}_{\overline{\rho}}$ the image of $\mathrm{Stab}_{\overline{\rho}}\cap\Upsilon$ 
in $\Sy_{\Psi}$. Then $\mathrm{St}_{\overline{\rho}}$ is an open subgroup of $\Sy_{\Psi}$, i.e., 
$\mathrm{St}_{\overline{\rho}}\supseteq\Sy_{\Psi|T_{\overline{\rho}}}$ for a finite 
$T_{\overline{\rho}}\subset\Psi$ such that $\overline{k(T_{\overline{\rho}})}\supseteq 
L_{\overline{\rho}}$, so $H^0(G_{\Psi},W_{\overline{\rho}}\otimes_kC_{\overline{\rho}^{\vee}})$ 
is a smooth $k(\Psi)$-semilinear representation of $\mathrm{St}_{\overline{\rho}}$ with ``trivial'' 
restriction to $\Sy_{\Psi|T_{\overline{\rho}}}$, i.e., $H^0(G_{\Psi},W_{\overline{\rho}}\otimes_k
C_{\overline{\rho}^{\vee}})=H^0(G_{\Psi},W_{\overline{\rho}}\otimes_k
C_{\overline{\rho}^{\vee}})^{\Sy_{\Psi|T_{\overline{\rho}}}}\otimes_{k(T_{\overline{\rho}})}k(\Psi)$. 
By Lemma~\ref{triv-on-open-subgrp}, the $k(\Psi)\langle\mathrm{St}_{\overline{\rho}}\rangle$-module 
$H^0(G_{\Psi},W_{\overline{\rho}}\otimes_kC_{\overline{\rho}^{\vee}})$ is ``trivial'', i.e., 
$H^0(G_{\Psi},W_{\overline{\rho}}\otimes_kC_{\overline{\rho}^{\vee}})=
H^0(G_{\Psi},W_{\overline{\rho}}\otimes_kC_{\overline{\rho}^{\vee}})^{\Sy_{\Psi|T_{\overline{\rho}}}}
\otimes_{k(\Psi)^{\Sy_{\Psi|T_{\overline{\rho}}}}}k(\Psi)$, 
and therefore, $V_O=H^0(G_{\Psi},W_{\overline{\rho}}\otimes_kC_{\overline{\rho}^{\vee}})
^{\mathrm{St}_{\overline{\rho}}}\otimes_{k(\Psi)^{\mathrm{St}_{\overline{\rho}}}}
k(\Psi)\langle\Sy_{\Psi}/\mathrm{St}_{\overline{\rho}}\rangle$ is a direct sum of 
several copies of $k(\Psi)\langle\Sy_{\Psi}/\mathrm{St}_{\overline{\rho}}\rangle$. \end{proof} 

\begin{lemma} \label{triv-on-open-subgrp} Let $U\subseteq\Sy_{\Psi}$ and $U'\subseteq U$ 
be open subgroups, $V$ be a smooth $k(\Psi)\langle U\rangle$-module such that 
$V=V^{U'}\otimes_{k(\Psi)^{U'}}k(\Psi)$. Then $V=V^U\otimes_{k(\Psi)^U}k(\Psi)$. \end{lemma} 
\begin{proof} It suffices to show that 
any cyclic $k(\Psi)\langle U\rangle$-submodule $V'$ of $V$ is a sum of submodules isomorphic 
to $k(\Psi)$. But $V'$ is a finitely generated $k(\Psi)\langle U'\rangle$-module, since 
$\Sy_{\Psi}$ is `{\sl Roelcke precompact}': $V$ is a quotient of 
$k(\Psi)\langle U/U''\rangle=\bigoplus_O(\bigoplus_{x\in O}k(\Psi)\cdot x)$ for an open subgroup 
$U''\subset U$, where $O$ runs over the (finite) set of $U'$-orbits on the set $U/U''$. 
As the finitely generated $k(\Psi)\langle U'\rangle$-module $V'$ is a sum of copies of 
$k(\Psi)$, it is finite-dimensional over $k(\Psi)$. By Lemma~\ref{open-subgrps-descr}, 
$U$ admits a normal subgroup of finite index of type $\Sy_{\Psi|T}$ for a finite 
$T\subset\Psi$. By Theorem~\ref{smooth-simple}, $V'=(V')^{\Sy_{\Psi|T}}\otimes_{k(T)}k(\Psi)$, 
and therefore, $V=V^{\Sy_{\Psi|T}}\otimes_{k(T)}k(\Psi)$; 
by Proposition~\ref{Satz1}, $V^{\Sy_{\Psi|T}}=V^U\otimes_{k(\Psi)^U}k(T)$. \end{proof} 

\begin{remark} It is not true in general that $H^0(G_{\Psi},W)$ is injective, even if $W=V\otimes C$ 
for a $\mathbb Q$-linear smooth representation $V$ of $G$. E.g., if $V$ is the kernel of the degree 
morphism $\mathbb Q\langle C\setminus k\rangle\to\mathbb Q$ then one has an exact sequence 
$0\to H^0(G_{\Psi},W)\to H^0(G_{\Psi},C\langle C\setminus k\rangle)\to k(\Psi)\to 0$, 
which is not split, since $H^0(G_{\Psi},C\langle C\setminus k\rangle)^{\Sy_{\Psi}}=
H^0(G,C\langle C\setminus k\rangle)=0$. 
\end{remark} 

\subsection{Smooth semilinear representations of symmetric groups with quasi-trivial connections} 
For field extensions $K|k$ and $L|K$, a $K$-vector space $V$ and a connection 
$\nabla:V\to V\otimes_K\Omega_{K|k}$, denote by 
$\nabla_L:V\otimes_KL\to V\otimes_K\Omega_{L|k}$ the unique extension of $\nabla$. 
If $V$ is endowed with an action of a group $H$ then a connection on $V$ is called 
a $H$-{\sl  connection} if it commutes with the $H$-action. 

A connection $\nabla:V\to V\otimes_K\Omega_{K|k}$ is called {\sl trivial} 
(resp., {\sl quasi-trivial}) if the natural map $\ker\nabla\otimes_kK\to V$ 
is surjective (resp., if $\nabla_{\overline{K}}$ is trivial). 

If $k$ is algebraically closed and $H$ is a group of automorphisms of $K$ then the functor 
of horizontal sections $\ker\nabla_{\overline{K}}:(V,\nabla)\mapsto\ker\nabla_{\overline{K}}$ 
is an equivalence of categories 
\[\begin{array}{ccc}\left\{\begin{array}{c}\mbox{smooth}\\ 
\mbox{$k\langle\widetilde{H}\rangle$-modules}\end{array}\right\}&
\mathop{\stackrel{\sim}{\longleftarrow}}\limits_{\ker\nabla_{\overline{K}}}&
\left\{\begin{array}{c}\mbox{smooth $K\langle H\rangle$-modules with}\\ 
\mbox{quasi-trivial $H$-connection over $k$}\end{array}\right\}\end{array},\]
where $\widetilde{H}$ is the group of all field automorphisms of $\overline{K}$ inducing elements 
of $H$ on $K$, so $\widetilde{H}$ the extension of $H$ by 
$\mathrm{Gal}(\overline{K}|K)$. The inverse 
functor is given by $W\mapsto((W\otimes_k\overline{K})^{\mathrm{Gal}(\overline{K}|K)},\nabla_W)$, 
where $\nabla_W$ is restriction of the connection on $W\otimes_k\overline{K}$ vanishing on $W$. 

Consider the following diagram of functors. 
\[\begin{array}{ccccc} \left\{\begin{array}{c}\mbox{smooth}\\ 
\mbox{$k\langle G\rangle$-modules}\end{array}\right\}&
\mathop{\stackrel{\sim}{\longrightarrow}}\limits_{\otimes_kC}&
\left\{\begin{array}{c}\mbox{smooth}\\ 
\mbox{$C\langle G\rangle$-modules with}\\ 
\mbox{trivial $G$-connection}\end{array}\right\}&
\stackrel{\mbox{for}}{\longrightarrow}&
\left\{\begin{array}{c}\mbox{smooth}\\ 
\mbox{$C\langle G\rangle$-modules}\end{array}\right\}\\ 
\downarrow\lefteqn{\mbox{restriction}}&&\downarrow\lefteqn{H^0(G_{\Psi},-)}
&&\downarrow\lefteqn{H^0(G_{\Psi},-)}\\ 
\left\{\begin{array}{c}\mbox{smooth}\\ 
\mbox{$k\langle\Upsilon\rangle$-modules}\end{array}\right\}&
\mathop{\stackrel{\sim}{\longleftarrow}}\limits_{\ker\nabla_C}&
\left\{\begin{array}{c}\mbox{smooth $k(\Psi)\langle\Sy_{\Psi}\rangle$-modules}\\ 
\mbox{with quasi-trivial}\\ 
\mbox{$\Sy_{\Psi}$-connection}\end{array}\right\}&
\stackrel{\mbox{for}}{\longrightarrow}&
\left\{\begin{array}{c}\mbox{smooth}\\ 
\mbox{$k(\Psi)\langle\Sy_{\Psi}\rangle$-modules}\end{array}\right\}\\
\downarrow\lefteqn{\mbox{restriction}}&&\\
\left\{\begin{array}{c}\mbox{smooth}\\ 
\mbox{$k\langle G_{\Psi}\rangle$-modules}\end{array}\right\}&&\end{array}\]
By \cite[Lemma 4.14]{max}, restriction to $\mathcal I_G(k)$ of the composition of 
the upper row is fully faithful. 

It is explained in \cite[\S4.5]{RMS}, that some conjectures (a conjectural relation to Chow groups 
of 0-cycles of projective generators of the category $\mathcal I_G(k)$ and the motivic conjectures) 
imply that there are only finitely many (or no) isomorphism classes of simple objects of 
$\mathcal I_G(k)$ containing a given irreducible representation of $G_{\Psi}$. In fact, (in the 
spirit of Howe--Bushnell--Kutzko--et al.) one can expect that any simple object of 
$\mathcal I_G(k)$ is determined uniquely by its restriction to $G_{\Psi}$. 

From this point of view, the smooth $k(\Psi)$-semilinear representations 
of $\Sy_{\Psi}$ with quasi-trivial $\Sy_{\Psi}$-connection should carry 
interesting information on the corresponding simple objects of $\mathcal I_G(k)$. 

\vspace{5mm}

\noindent
{\sl Acknowledgements.} {\small The article was prepared within the framework 
of the Academic Fund Program at the National Research University Higher School 
of Economics (HSE) in 2015--2016 (grant no. 15-01-0100) and supported within 
the framework of a subsidy granted to the HSE by the Government of the Russian 
Federation for the implementation of the Global Competitiveness Program.}

\end{document}